\documentclass[10pt]{iopart}
\usepackage{iopams}
\usepackage{setstack}  
\usepackage{amsthm}
\usepackage{graphicx}

\newtheorem{theorem}{Theorem}[section]
\newtheorem{lemma}[theorem]{Lemma}
\newtheorem{corollary}[theorem]{Corollary}
\newtheorem{proposition}[theorem]{Proposition}
\newtheorem{definition}[theorem]{Definition}
\newtheorem{algorithm}[theorem]{Algorithm}
\newtheorem{remark}[theorem]{Remark}

\usepackage{enumitem} 
\usepackage{subfigure} 
\usepackage{url}

\newcommand{\opA}{\mathcal{A}}
\newcommand{\opB}{\mathcal{B}}
\newcommand{\opPP}{\mathcal{P}} 
\newcommand{\opL}{\mathfrak{L}}
\newcommand{\opR}{\mathfrak{R}}
\newcommand{\opP}{\mathfrak{P}}
\newcommand{\opI}{\mathfrak{I}}
\newcommand{\opW}{\mathfrak{W}}
\newcommand{\opX}{\mathfrak{X}}
\newcommand{\opE}{\mathrm{E}}

\newcommand{\X}{\mathcal{X}}
\newcommand{\Y}{\mathcal{Y}}
\newcommand{\Z}{\mathcal{Z}} 

\newcommand{\R}{\mathbb{R}}
\newcommand{\N}{\mathbb{N}}
\newcommand{\D}{\mathbb{D}}
\renewcommand{\S}{\mathbb{S}}
\newcommand{\E}{\mathbb{E}}

\usepackage{xcolor}

\newcommand{\Sb}{\mathbf{s}}
\newcommand{\Db}{\mathbf{d}}
\newcommand{\Xb}{\mathbf{x}}
\newcommand{\Yb}{\mathbf{y}}
\newcommand{\Zb}{\mathbf{z}}

\usepackage{comment}

\graphicspath{{Tikz/}{Reconstructions/}{Reconstructions/g_1_noise_free/}{Reconstructions/g1_noisy_2.4/}{Reconstructions/g2_noise_free/}{Reconstructions/g2_pseudo_noise_free/}}

\begin{document}

\title[]{Imaging based on Compton scattering: model uncertainty and data-driven reconstruction methods}

\author{Janek Gödeke$^\dag$ and Ga\"el Rigaud$^\ddag$}
\address{$^\dag$ Zentrum f\"ur Technomathematik, University of Bremen, Germany\\
$^\ddag$ Department of Mathematics, University of Stuttgart, Germany}
\ead{gael.rigaud@mathematik.uni-stuttgart.de}
\vspace{10pt}
\begin{indented}
\item[]January 2022
\end{indented}

\begin{abstract}
The recent development of scintillation crystals combined with $\gamma$-rays sources opens the way to an imaging concept based on Compton scattering, namely Compton scattering tomography (CST). The associated inverse problem rises many challenges: non-linearity, multiple order-scattering and high level of noise. Already studied in the literature, these challenges lead unavoidably to uncertainty of the forward model. This work proposes to study exact and approximated forward models and develops two data-driven reconstruction algorithms able to tackle the inexactness of the forward model. The first one is based on the projective method called regularized sequential subspace optimization (RESESOP). We consider here a finite dimensional restriction of the semi-discrete forward model and show its well-posedness and regularisation properties. The second one considers the unsupervised learning method, deep image prior (DIP), inspired by the construction of the model uncertainty in RESESOP. The methods are validated on Monte-Carlo data.   
\end{abstract}



\section{Introduction}

\par
At first a tool for visualizing the inside of the human body using X-rays by the upcoming of
Computerized Tomography (CT), the need for imaging affects nowadays astrophysics, homeland security, landscape and environment monitoring and of course manufacturing processes to cite only a few. This success is made possible by the technological progress in terms of detection -- cameras, crystals, etc -- but also in terms of computing and storage capacities. \\

\par 
Computerized Tomography (CT) is a well-established and widely used technique which images an object by exploiting the properties of penetration of the x-rays. Due to the interactions of the photons with the atomic structure, the matter will resist the propagation of the photon beam of energy $E$ and intensity $I(\mathbf{x})$ according to the well-known Beer-Lambert law 
\begin{equation}\label{eq:BeerLambert}
I(\mathbf{y}) = I(\mathbf{x}) e^{-\int_{\mathbf{x} \to \mathbf{y}} \mu} ,
\end{equation}
where $\mu$ stands for the lineic attenuation coefficient and $\mathbf{x} \to \mathbf{y}$ denotes the straight line $\{\mathbf{x}+t(\mathbf{y}-\mathbf{x}), \ t\in [0,1]\}$. To interpret the measurement of the intensity in a CT-scan is then possible with the help of the Radon transform in 2D and the X-ray transform in 3D, which maps the attenuation map $\mu(x)$ into its line integrals, \textit{i.e.}
\begin{equation}\label{eq:Radontransform}
\ln \frac{I(\mathbf{s},\theta)}{I(\mathbf{d},\theta)} = \opR \mu (p,\theta) = \int_\Omega \mu(x) \delta(p - x\cdot \theta) \mathrm{d}x
\end{equation}
with $(p,\theta)  \in \mathbb{R} \times S^1$ and where $\mathbf{s}$ and $\mathbf{d}$ stand for the position of the source and of the detection point. We refer to \cite{Natterer} for more information. \\

\par The energy constitutes an important variable made accessible by the recent development of scintillation crystals and semi-conductors detectors \cite{Knoll}. Currently the energy is exploited in multi-spectral CT as a supplementary variable split into several channels delivering a precious information on the attenuation coefficient at different energy levels. We refer to \cite{AM_76,PRLYM_2009,Shefer2013,TM_2015,MLLF_2015,GG_2017,Fredenberg_2018}. However the recently achieved energy resolution, more precisely the FWHM, of the current scintillation crystals opens the way to consider the energy as a reliable dimension along with viewpoints and detector positions. In particle physics, the question of the energy intersects with Compton scattering. Indeed, when one focuses on the physics between the matter and the photons, four types of interactions come out: Thomson-Rayleigh scattering, photoelectric absorption, Compton scattering and pair production. In the classic range of applications of the x-rays or $\gamma$-rays, $[50,1000]$ keV, the photoelectric absorption and the Compton scattering are the dominant phenomena which leads to a model for the lineic attenuation factor due to Stonestrom et al. \cite{Stonestrom} which writes
\begin{equation}\label{eq:stonestrom}
\mu(\mathbf{x},E) = E^{-3} \lambda_{PE}(\mathbf{x}) + \sigma(E) f(\mathbf{x})
\end{equation} 
where $\lambda_{PE}$ is a factor depending on the materials and symbolizing the photoelectric absorption, $\sigma(E)$ the total-cross section of the Compton effect at energy $E$ and $f$ the electron density (generally noted $n_e$) at $\mathbf{x}$. \\

\par 
The Compton effect stands for the collision of a photon with an electron. The photon transfers a part of its energy $E_0$ to the electron. The electron suffers then a recoil and the photon is then scattered of an (scattering) angle $\omega$ with the axis of propagation. The energy of the photon after scattering is expressed by the Compton formula \cite{Compton_23},
\begin{equation}\label{eq:Compton_formula}
E = \frac{E_0}{1+\frac{E_0}{mc^2}(1-\cos\omega)} =: \opE(\omega),
\end{equation} 
where $mc^2 = 511$ keV represents the energy of an electron at rest.  Measuring accurately the variations of the energy can thus be interpreted as scattering events characterized geometrically by the scattering angle which is the foundation of Compton scattering tomography (CST), see \cite{ABE_11,AHD_90,BC_03,BCGLR_02,CV_73,EMBR_98,FC_71,GKAD_05,Cesareo_02,Hussein,Arendtsz,Guzzardi,HH_2010,Norton_94,Nguyen2011,Nguyen2017,Rigaud_SIIMS_17,WebberLionheart}.  \\

\subsection{Spectral data}

\par Given a monochromatic $\gamma$-ray source $\Sb$ of energy $E_0$ and an energy-resolved detector $\Db$, the illumination of a specimen represented by its attenuation map $\mu$ leads by the Compton effect to a polychromatic response measured at $\Db$. This would also hold for a polychromatic source as studied in \cite{kuger2020joint,KugerRigaud21} but for the sake of simplicity we consider in this work only monochromatic sources. Assuming only Compton scattering and photoelectric absorption events, we can decompose the spectrum $\mathrm{Spec}(E,\Db,\Sb)$ measured at a detector $\Db$ with energy $E$ as follows
\begin{equation}\label{eq:spectrum}
\mathrm{Spec}(E,\Db,\Sb) = \sum_{i\in \N} g_i(E,\Db,\Sb).
\end{equation}
The data $g_i$ stands for the measured radiation without scattering events for $i=0$ and after $i$-scattering events for $i>0$. The ballistic data $g_0$ can be understood as the intensity $I(\Db,\theta)$ in eq. (\ref{eq:Radontransform}). Widely studied in 2D \cite{kuger2020joint} and 3D \cite{Rigaud21}, the first-order scattered radiation can be modeled by weighted circular or toric Radon transform and shares similarities with $g_0$ in particular in terms of mapping properties. More generally, $g_i$, $i\geq 1$, can be seen as a special case of the integral transforms
$$
\opL_i (\mu,\bar{f}) (E,\Db,\Sb) := \int_{\Omega^i} \bar{f} (\Zb) \; k_i(\mu;\Zb,E,\Db,\Sb) \; \mathrm{d}\Zb, \quad \bar{f} = \left(\underbrace{f \otimes \ldots \otimes f}_{i \text{ times}}\right)
$$
with  $k_i(\cdot)$ 
a singular kernel, $\Omega \subset \R^d$, $d=2,3$ and $(\E,\D,\S)$ the domain of definition of $(E,\Db,\Sb)$.  The complexity to handle $k_i$ computationally, already for $i=2$ studied in \cite{kuger2020joint,Rigaud21}, combined with its nonlinearity with respect to $f$ (already for the first-order scattering as $\mu$ is a function of $f$), makes the use of multiple-order scattering intractable in practice, at least with the current level of technology.  The exploitation of scattering in imaging is thus extremely challenging at theoretical and computational levels and solving (\ref{eq:spectrum}), \textit{i.e.} finding $f$ from $\mathrm{Spec}$ given the scattering model $\opL_1$, will lead to a large model inexactness.\\

\par Therefore, one needs appropriate reconstruction methods able to tackle this limitation of the model. Two approaches appear suited and are considered in this work. The first one is the RESESOP (regularized sequential subspace optimization) developed in \cite{Blanke_20} for dealing with model inexactness. The principle of this method is to split the inverse problem into subproblems and to relax the solution set for each using stripes instead of hyperplanes. The thickness of the stripes is then controlled by a parameter of model uncertainty. The second approach is the widely used deep image prior (DIP) unsupervised learning technique, which was presented in \cite{Lempitsky2018} for denoising and inpainting problems. The reason to use this approach is twofold: (i) it does not require datasets which are at the moment inexistant for CST, (ii) it provides a very flexible architecture while sharing interesting properties from optimization.  \\

\par Studied in \cite{Rigaud21}, the shape and disposition of the detector array is important to the structure of the forward models. We denote by $\S_{\frac12}^1(d) \subset \R^d$ the half-sphere of dimension $d-1$ and parametrized by angles $(\alpha_1,\ldots,\alpha_{d-1})$. 
We define the set of detector positions defined by $\Sb$ and $t$ as
\begin{equation}\label{eq:def_D}
\D(t) := \left\{\Db = t(\alpha_1)\; \theta(\alpha_1,\ldots,\alpha_{d-1}), \  \theta \in \S_{\frac12}^1(d) \right\}
\end{equation}
with $t$ a smooth function. 
For the implementation, we consider here the case $t(\alpha) = \cos \alpha$ which characterizes the sphere passing through $\mathbf{0}$ and will denote below $\D(\cos)$ by $\D$ for the sake of readibility.\\

\subsection{Outline and contributions}

\par 
The paper is organized as follows. Section \ref{sec:forward} recalls the forward models associated to the first- and second-order scattered part of the spectrum, see \cite{Rigaud21,kuger2020joint}. We study the nonlinearity of the first term and discuss how standard algorithms could be exploited at the cost of large computation costs and favorable prior information. More flexible, a linear approximation of the forward operators leads to interesting mapping properties, see also \cite{Rigaud21,kuger2020joint}, and is suited for reconstruction strategies. Due to the complexity of the second-order part, we focus the inverse problem on the first order part. However, this approximation implies a strong model uncertainty in particular when incorporating the second-order scattering. In order to solve the spectral inverse problem (\ref{eq:spectrum}) with such an inaccuracy, we propose first to adapt in Section \ref{sec:resesop} the RESESOP  method. In \cite{Blanke_20}, the authors proved that the proposed RESESOP-Kaczmarz is a regularization method for the SESOP (sequential subspace optimization) method with exact model, see Theorem \ref{thm:RESESOP_regularization}. The spectral problem is reformulated first as semi-discrete, and then as fully discrete, more precisely we consider a finite dimensional restriction of the solution space. It follows by Corollary \ref{coro:regularization_fixed_j} that the RESESOP method adapted to the fully discrete problem regularizes the semi-discrete one. Furthermore, the constructed solution for the fully-discrete problem converges to the minimum norm solution of the semi-discrete problem for a suitable start iterate, see Theorem \ref{thm:full_semi_regularization}. Inspired by the RESESOP approach, we then derive in Section \ref{sec:DIP} an appropriate loss function for a DIP algorithm. Simulation results are presented in Section \ref{sec:simu} for synthetic data and Monte-Carlo data for the second-order scattered radiation. A conclusion ends the manuscript.

\section{Formulation of the mathematical problem}\label{sec:forward}

As explained in the Introduction, the measured spectrum is the sum of the primary radiation and of the scattered radiation of different orders. From a physical point of view, the lower the energy the lower is the probability of a scattering event. It follows that high-order scattering events, typically $\geq 3$, represent a marginal part of the scattered radiation and by the stochastic nature of the emission of photons will be highly noisy. To reflect this physical point of view, we consider that 
\begin{equation}\label{eq:spectral_IP}
\mathrm{Spec} = g_0 + g_1 + g_2 + \epsilon     
\end{equation}
with $\epsilon$ a noisy perturbation. In this section, we recall the modelling of the first- and second-order scattered radiation, their properties and detail the computation of the spectral data for a specific scanning architecture. The section ends with the presentation of a general reconstruction strategy.

\subsection{The forward problem }

As proven in \cite{RH_2018,Rigaud_SIIMS_17}, the first-order scattered radiation $g_1$ can be modelled by the integration of the electron density $f$ along spindle tori (in 3D) or circular-arcs (in 2D) expressed as
$$
\mathfrak{T}(\omega,\Db,\Sb) = \left\{ \Xb \in \mathbb{R}^d \ : \ \sphericalangle (\Xb-\Sb,\Db - \Xb) =  \omega \right\}, \qquad d=2,3.
$$

\begin{figure}
    \centering
    \includegraphics{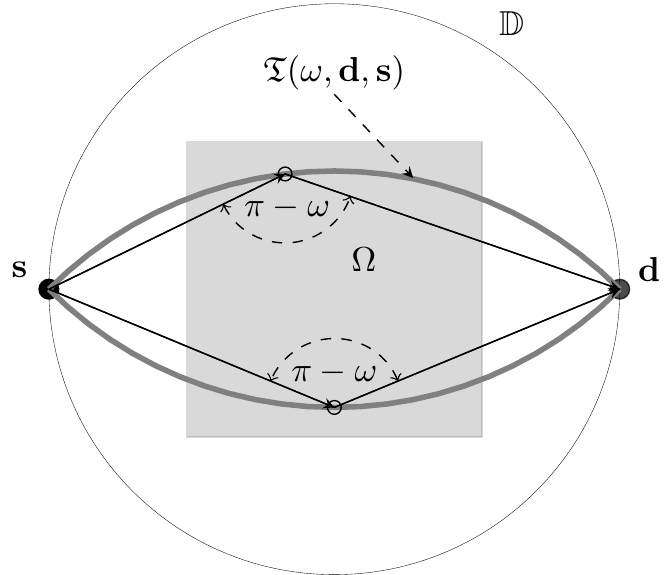}
    \caption{One time scattered photons arriving $\Db$ from $\Sb$ with energy $\mathrm{E}(\omega)$ have been scattered on $\mathfrak{T}(\omega, d,s)$, two opposite circular-arcs in 2D.}
    \label{fig:torus}
\end{figure}
For an illustration of the geometry of a circular-arc see Figure \ref{fig:torus}.
It follows that
\begin{equation}\label{eq:def_L1}\fl 
g_1 \sim \opL_1(\mu,f)(E,\Db,\Sb) := \int_{\Omega} \opW_1(\mu) (\Xb,\Db,\Sb) \ f(\Xb) \ \delta(E-\phi(\Xb,\Db,\Sb)) \ \mathrm{d}\Xb
\end{equation}
where $\opW_1 (\mu)$ quantifies the physical factors (attenuation and photometric dispersion) between $\Sb,\Xb$ and $\Db$, and  $\phi$ stands for the level-set function associated to the inside  (resp. outside) spindle torus when positive (resp. negative) and is given by 
\begin{equation}\label{eq:phase_phi}
\phi(\Xb,\Db,\Sb)  = \opE\left( \cot^{-1} \frac{\kappa(\Xb,\Db,\Sb) - \rho(\Xb,\Db,\Sb)}{\sqrt{1-\kappa^2(\Xb,\Db,\Sb)}} \right)
\end{equation}
where
\begin{equation}\label{eq:def_kappa_rho}
\kappa(\Xb,\Db,\Sb) = \frac{(\Xb-\Sb)}{\Vert \Xb-\Sb \Vert_2 } \cdot \frac{(\Db-\Sb)}{\Vert \Db-\Sb \Vert_2 } 
\quad \text{and} \quad 
\rho(\Xb,\Db,\Sb) = \frac{\Vert \Xb-\Sb \Vert_2}{\Vert \Db-\Sb \Vert_2}
\end{equation}
with $\Vert \cdot\Vert_2$  the euclidean norm.\\

\par Studied in \cite{Rigaud21,kuger2020joint}, the second-order scattered radiation $g_2$ can be represented, akin to $g_1$, by a nonlinear integral transform
\begin{equation}\label{eq:def_L2}\fl 
g_2 \sim \opL_2(\mu,\bar{f})(E,\Db,\Sb) := \int_{\Omega^2} \opW_2(\mu) (\Zb,\Db,\Sb) \ \bar{f}(\Zb) \ \delta(E-\psi(\Zb,\Db,\Sb)) \ \mathrm{d}\Zb
\end{equation}
where $\opW_2 (\mu)$ quantifies the physical factors between $\Sb,\Zb$ and $\Db$ with $\Zb = (\Xb,\Yb)$, $\bar{f} = f \otimes f$ and $\psi$ characterizes the locations of successive first- and second-order scattering events. We refer to \cite{Rigaud21,kuger2020joint} for more details.

\subsection{A look on the nonlinear problem}

\par The operators $\opL_1$ and $\opL_2$ are nonlinear w.r.t. $f$ and $\bar{f}$ respectively but also difficult to handle numerically. The Fr\'echet (or G\^ateaux) derivative is then essential in the construction of reconstruction schemes. Focusing on the first-order scattering, we can compute the corresponding Fréchet derivative where we neglect the photoelectric absorption in $\mu$ in eq. (\ref{eq:stonestrom}), \textit{i.e.} $\mu(\cdot,E) = \sigma(E) f$, and consider, for the sake of simplicity, the operator $\opL_1(f):= \opL_1(\mu,f) : \X\to\Y$, with $\X,\Y$ two suited Hilbert spaces equipped with their respective norms. Furthermore, the weight $\opW_1$ (see \cite{Rigaud21,kuger2020joint} for more details) can then be written as
$$
\opW_1(f) = C  \frac{\exp\left( - \opX f \right)}{\Vert \Xb-\Sb\Vert_2^2 \Vert \Db-\Xb\Vert_2^2 }, \quad C>0
$$
where $\opX$ denotes the X-ray transform applied on the electron density $f$ along the scattering path $\Sb$ to $\Xb$ and $\Xb$ to $\Db$ (see eq. (\ref{eq:BeerLambert})).

\begin{theorem}\label{thm:L1_Frechet}
$\opL_1$ is Fréchet-differentiable with
$$
(\opL_1)_f^\prime h(E,\Db,\Sb) = \int_{\Omega} \left[\left( \opW_1 \right)^\prime_f h (\Zb,\Db,\Sb)  f(\Zb) + \opW_1(f)(\Zb,\Db,\Sb)  h(\Zb)\right]   \;  \delta(E-\phi(\Zb,\Db,\Sb)) \; \mathrm{d}\Zb
$$
and $(\opL_1)_f^\prime$ is bounded for every $f\in L_2(\Omega)$ bounded.
\end{theorem}
\begin{proof}
The Fréchet derivative for $\opL_1$ is defined as 
$$
\opL_1 (f+h) = \opL_1 (f) +  (\opL_1)_f^\prime h + o(h).
$$

Inspecting $\opW_1$ closer, we note that it is the composition of a smooth function and of a linear operator which implies that $\opW_1$ is Fréchet-differentiable. Given the Fréchet derivative of $\opW_1$, it holds
$$
\opW_1(f+h) = \opW_1(f) + \left( \opW_1 \right)^\prime_f h + o(h)
$$
leading to 
\begin{eqnarray*}
\opL_1 (f+h) &=\opL_1 (f) 
+ \int_{\Omega} \left( \opW_1 \right)^\prime_f h(\Zb,\Db,\Sb)  \; f(\Zb) \;  \delta(E-\phi_1(\Zb,\Db,\Sb)) \; \mathrm{d}\Zb\\
&+ \int_{\Omega} \opW_1(f)(\Zb,\Db,\Sb)  \; h(\Zb) \;  \delta(E-\phi_1(\Zb,\Db,\Sb)) \; \mathrm{d}\Zb\\
&+ \int_{\Omega} h(\Zb) \; \left( \opW_1 \right)^\prime_f h(\Zb,\Db,\Sb)  \;   \delta(E-\phi_1(\Zb,\Db,\Sb)) \; \mathrm{d}\Zb\\
&+ \int_{\Omega} f(\Zb) \; o(h) \;  \delta(E-\phi_1(\Zb,\Db,\Sb)) \; \mathrm{d}\Zb.
\end{eqnarray*}
The linear part w.r.t. $h$ on the right handside, \textit{i.e.} its Fréchet derivative, reads now
$$
(\opL_1)_f^\prime h =  \int_{\Omega} \left[\left( \opW_1 \right)^\prime_f h (\Zb,\Db,\Sb)  f(\Zb) + \opW_1(f)(\Zb,\Db,\Sb)  h(\Zb)\right]  \;  \delta(E-\phi_1(\Zb,\Db,\Sb)) \; \mathrm{d}\Zb .
$$
For $\Omega$ compactly supported and for $f$ being bounded, it is clear that the linear operator $f(\Zb) \left( \opW_1 \right)^\prime_f + \opW_1(f) \opI$ is bounded and consequently the property holds for $(\opL_1)_f^\prime $.
\end{proof}

We observe that the computation of the Fr\'echet derivative of $\opL_1$, for instance within the Kaczmarz's method, would require the computation of $\left( \opW_1 \right)^\prime_f$ and $\opW_1(f)$ at each iterate which constitutes, especially in 3D, an expensive task.

Besides the computation cost, the Fr\'echet derivative needs to satisfy the so-called \textit{tangential cone condition} which would read as
$$
\Vert \opL_1(f+h) - \opL_1(f) - (\opL_1)_f^\prime h \Vert_\Y \leq c_t \Vert \opL_1(f+h) - \opL_1(f) \Vert_\Y ,
$$
with some constant $c_t < 1$, in order that most of the iterative schemes applied on $\opL_1$ converge.
Using the expression of $\opW_1$, it holds with symbolic notations
\begin{eqnarray*} 
&\fl \opL_1(f+h) - \opL_1(f) - (\opL_1)_f^\prime h \\
&\fl= \int_{\mathfrak{T}} \left[\opW_1(f+h) \; (f+h) - \opW_1(f)\; f - \opW_1(f) h - \left(\opW_1 \right)^\prime_f h \; f  \right]\\
&\fl= \int_{\mathfrak{T}} \left[\left(\opW_1(f+h) - \opW_1(f)\right)\; (f+h)  - \left(\opW_1 \right)^\prime_f h \; f  \right] \\
&\fl= \int_{\mathfrak{T}} \left[ \opW_1(f+h) \left(1 - \opW_1(-h)\right)\; (f+h)  + (\opX h) \opW_1(f) \; f  \right]\\
&\fl= \int_{\mathfrak{T}} (-\opX h)\left[ \opW_1(f+h) \; (f+h)  +  \opW_1(f) \; f  \right] -  \int_{\mathfrak{T}} \sum_{n=2}^\infty \frac{(\opX h)^n}{n!}  \opW_1(f+h)\; (f+h).
\end{eqnarray*}
We observe that the tangential cone condition might not hold for "large" $h$ as the second term explodes for $h$ large. Therefore, dealing with the nonlinear problem might require an \textit{a priori} initial value close to the solution which is not always possible to guarantee. 

\subsection{Linear approximations and mapping properties}

This is the reason why it is relevant to split the dependency on $f$ and therefore study instead linear approximations $\opL_1(\mu^{*},\cdot)$ and $\opL_2(\mu^{*},\cdot)$ with $\mu^{*}$ a \textit{known a priori} smooth approximation to the original $\mu$. Such approximations have the following properties on the Sobolev scale.

\begin{definition}
Let $\Omega \subset \R^d$, the Sobolev space of order $\alpha$, noted $H^\alpha$, is defined as
$$
H^\alpha (\Omega) := \left\{ f\in L_2(\Omega) \ : \ (1+|\xi|^2)^{\alpha/2} \hat{f} \in L_2(\R^d) \right\}
$$
with $\hat{f}$ the Fourier transform of $f$.
We denote by $H_0^\alpha(\Omega) \subset H^\alpha(\Omega)$ the Sobolev space of order $\alpha$ with functions vanishing at the boundaries of $\Omega$.
Furthermore, we define the Sobolev space of order $\alpha$ of locally square-integrable functions by
\begin{equation*}
    H_{loc}^\alpha(\Omega) := \big \{f\in L_{2,loc}(\Omega): \varphi f\in H^\alpha_0(\Omega) \textup{ for all } \varphi\in C^\infty(\Omega) \big \}.
\end{equation*}
\end{definition}

\begin{theorem}[\cite{Rigaud21,kuger2020joint}]\label{thm:smoothness_props}
We let $h \in C^\infty(\Omega)$ with $\Omega \subset \R^d$, $d=2,3$. Then, for the detector set $\D$ defined in eq. (\ref{eq:def_D}) where the domain of the parameters $(\alpha_1,\ldots,\alpha_{d-1})$ is open and $\Sb\in\D$ fixed, the operators $\opL_1(h,\cdot)$ and $\opL_2(h,\cdot)$ are Fourier integral operators of order 
$$
\tau_1 := \frac{1-d}{2} \quad \text{and} \quad \tau_2 := \frac{2-3d}{4}
$$ 
respectively and it exists $\E \subset \R^+$ where they are continuous mapping from $H_0^\alpha(\Omega)$ to $H_{loc}^{\alpha-\tau_1}(\E,\D)$ and $H_0^\beta(\Omega^2)$ to $H_{loc}^{\beta-\tau_2}(\E,\D)$ respectively for all $\alpha \in \R$ and $\beta \in \R^+$.
\end{theorem}
\begin{proof}
See \cite{Rigaud21} for $d=3$ and \cite{kuger2020joint} for $d=2$.
\end{proof}

\par In the next sections, we consider a semi-discrete and a fully discrete setting for the spectral data in order to better reflect data acquisition. To this end, we consider a set of source and detector positions $(\Sb_k,\Db_k)_{k=1,\ldots,K}$ as well as a set of energy $(E_p)_{p=1,\ldots,P}$. The sampling step on the energy will in practice depend on the energy resolution (called FWHM) of the detector.  

\par 
We also aim to connect these both settings in terms of representation. 
However, similarly to the semi-discrete Radon transform, see e.g. in \cite[Chapter 6.3.]{Rieder}, we face the problem that sampling $\opL_1(\mu,f)$ on a finite set is not well-defined for arbitrary $f\in L^2(\Omega)$, as its equivalence class may not have a continuous representative. Indeed, for $\Sb$ fixed, the operator 
\begin{eqnarray*}
    \opL_1: &L_2(\Omega) &\longrightarrow H_{loc}^{-\tau_1}(\E,\D)\\
   &f &\longmapsto \opL_1(\mu,f)(\cdot,\cdot,\Sb)
\end{eqnarray*}
is not well-defined in the semi-discrete setting, as $\opL_1(\mu,f)$ would have no continuous representative, and therefore discretizing $(E,\Db,\Sb)$ would be improper regarding the continuous case. However, we can exploit an embedding property for the Sobolev spaces. 
Before stating this property, we recall some geometric concepts.  
\begin{definition}{\cite[Section 2.1.]{Wloka}}
A \emph{cone} with vertex at $x\in \R^d$ is a set of the type
\begin{equation*}
    C(x, r, U) := B_r(x) \cap \{\lambda(y-x): y\in U, \lambda>0\},
\end{equation*}
where $B_r(x)$ is the open ball around $x$ with radius $r$ and $U$ is an open, non-empty subset of $\R^d$.
\end{definition}

\begin{definition}{\cite[Definition 2.2.]{Wloka}}\label{def:cone_property}
A set $\Omega'\subseteq \R^d$ is said to have the \emph{cone property} if there exists a cone $C_0$ in $\R^d$, such that for all $x\in \overline{\Omega'}$ there is a cone $C(x, r, U)\subseteq \Omega'$ with vertex at $x$, which is congruent to $C_0$. That is, $C_0$ and $C_x$ must be equal up to rotation and translation.
\end{definition}

The embedding theorem below is a special case of \emph{Sobolev's Lemma}, see \cite[Theorem 6.2.]{Wloka}. Therein, the set $\Omega'$ does not only need to satisfy the cone property, but also the \emph{segment property} in \cite[Definition 2.1.]{Wloka}. However, it is mentioned in \cite{Wloka} that for bounded $\Omega'$ the cone property is sufficient.

\begin{theorem}{\cite[Corollary 6.1]{Wloka}, \cite{Adams}} \label{thm:sobolev_embedding}
Let $\Omega'\subset \mathbb{R}^n$ be a bounded region and have the cone property. Then the Sobolev spaces $H_0^s(\Omega')$ and $H^s(\Omega')$ are continuously embedded into $C^m(\overline{\Omega'})$ for all $s > m + n/2$.
\end{theorem}

In order to exploit this result, it is important to relax the locality constraint in Theorem \ref{thm:smoothness_props}. To this end, we consider a suited smooth cut-off $\chi$ which vanishes at the boundaries of $(\E,\D)$ such that $\chi \opL_1(h,f) \in H_{0}^{\alpha-\tau_1}(\E,\D)$ for $h \in C^\infty(\Omega)$, $f \in H_0^\alpha(\Omega)$. Choosing $k=0$, it follows that the operator 
\begin{eqnarray*}
    \chi\opL_1: &H_0^\alpha(\Omega) &\longrightarrow H_{0}^{\alpha-\tau_1}(\E,\D)\\
    &f &\longmapsto \chi \opL_1(\mu,f)(\cdot,\cdot,\Sb)
\end{eqnarray*}
has a continuous representative for 
$$
\alpha - \frac{1-d}{2}> \frac{d}{2} \quad \Leftrightarrow \quad \alpha > \frac{1}{2}.
$$
The result holds similarly for $\opL_2$.
Therefore, assuming $\mu \in C^\infty(\Omega)$ and $f \in H_0^\alpha(\Omega)$ with $\alpha > 1/2$, we can now define the forward operators for  the semi-discrete first-order and second-order scattering by
\begin{eqnarray*}
    \opL_1^\mu: &H_0^\alpha(\Omega) &\longrightarrow \R^{P \times K}\\
    &f &\longmapsto \left(\chi \opL_1(\mu,f)(E_p,\Db_k,\Sb_k)\right)_{p=1,\ldots,P, \  k=1,\ldots,K}\\
    \opL_2^\mu : &H_0^\alpha(\Omega^2) &\longrightarrow \R^{P \times K}\\
    &\bar{f} &\longmapsto \left(\chi \opL_2(\mu,\bar{f})(E_p,\Db_k,\Sb_k)\right)_{p=1,\ldots,P, \  k=1,\ldots,K}. \\
\end{eqnarray*}
Letting aside the ballistic radiation $g_0$ which contributes in only one value at $E_0$ in the spectrum, the spectral problem (\ref{eq:spectral_IP}) becomes then 
\begin{equation}\label{eq:IP_sd}
\opL_1^\mu  f + \opL_2^\mu \bar{f}  = \mathrm{Spec}.   
\end{equation}
Most reconstruction techniques require the computation of the adjoint operator, here of $\opL_1^\mu$ and $\opL_2^\mu$, and consider the topology of the $L_2$-space in order to take into account perturbations in the measurement. However, akin to the semi-discrete Radon transform, see \cite{Rieder}, the adjoint of $\opL_1^\mu$ is not continuous w.r.t. the $L_2$-topology and its computation in the $H^\alpha$-topology can be a hard analytic and computational task. A way to circumvent this obstacle is to restrict the domain space to a finite dimensional subspace allowing us to use for instance the $L_2$-topology, by equivalence of the norms.
Therefore, we consider for the implementation the fully-discrete case which can be formulated by restricting the forward domain space into a subspace  $\X_j \subset H_0^\alpha(\Omega)$ with $\mathrm{dim}(\X_j) < \infty$, \textit{i.e.} 
\begin{eqnarray*} 
     \left(\opL_1^\mu \right)_j: & \X_j  &\longrightarrow \R^{P \times K}\\
     &f_j &\longmapsto \opL_1^\mu f_j  \\
     \left(\opL_2^\mu \right)_j : & \X_j \times \X_j  &\longrightarrow \R^{P \times K}\\
     &\overline{f_j} &\longmapsto \opL_2^\mu \overline{f_j}.
\end{eqnarray*}
Since $\X_j$ is finite dimensional, $\left(\opL_1^\mu \right)_j$ and $\left(\opL_2^\mu \right)_j$ are bounded with respect to the $L_2$-norm and more standard approaches can be used to solve 
\begin{equation}\label{eq:IP_fd}
\left(\opL_1^\mu \right)_j f_j + \left(\opL_2^\mu \right)_j \overline{f_j}  = \mathbf{g}_{1} + \mathbf{g}_{2}    
\end{equation}
where $\mathbf{g}_i$, $i=1,2$, denotes the sampled version of $g_i$. An interesting question is how to relate the solution to all subproblems (\ref{eq:IP_fd}) to the solution of (\ref{eq:IP_sd}). This is answered in Section \ref{sec:resesop}.

\subsection{Model uncertainty and reconstruction strategies}

Focusing on the Compton part, the spectral problem (\ref{eq:spectral_IP}) can be reformulated with the fully discrete setting in eq. (\ref{eq:IP_fd}) by
$$
\text{Find $f$ from $\mathbf{Spec}$ with  $\Vert\left(\opL_1^\mu \right)_j f_j + \left(\opL_2^\mu \right)_j \overline{f_j}  -  \mathbf{Spec} \Vert_2 \leq \epsilon, \ j\in \N$}
$$
in which $\mathbf{Spec} \in \R^{P \times K}$ is the sampled version of $\mathrm{Spec}$. 

\begin{remark}
Using $g_0$ in the reconstruction process is sensible. For instance, it is possible to reconstruct under sparsity constraints a first approximation of the attenuation map which can help to refine the forward model, in particular the weight functions, see \cite{KugerRigaud21,kuger2020joint}. However, we discarded this part of the spectrum as we wanted  to stress the model uncertainty using a weaker a priori of the attenuation map.
\end{remark}

This inverse problem is in particular challenging regarding the following two aspects:
\begin{itemize}
    \item \textbf{complexity:} 
    while the computational cost regarding $\opL_1^mu$ is similar to the one of the semi-discrete Radon transform   (assuming the weight function is precomputed),  evaluating $\opL_2^\mu$ is much more expensive. Given a grid of $N^d$ elements, then the complexity of $\opL_1^\mu$ is of order $O(N^{d} \times J \times K)$ while $\opL_2^\mu$ is of order $O(N^{2d}\times J \times K)$. In 2 dimensions and $N=100$, it means that the computation of the second order scattering is 10000 times more expensive! \\
    This represents an important obstacle which encourages us to focus on the first-order scattered radiation and forces us to use different simulators such as a Monte-Carlo approach for the second-order. 
    \item \textbf{model uncertainty:} the linearization of the forward models by assuming a prior attenuation map $\mu^*$  leads to an inaccuracy in the model \textit{i.e.} we have with some $\eta_{1j}>0$
    $$
    \Vert \left(\opL_1^\mu \right)_j f_j  - \left(\opL_1^{\mu^*} \right)_j f_j  \Vert_2 \leq \eta_{1j}.
    $$
    The issue of the model uncertainty further increases when focusing on the first-order scattering as proposed above. In this case, the second-order (and larger order in practice) has to be treated as model error as well, which yields 
    $$
    \Vert \left(\opL_1^\mu \right)_j f_j   + \left(\opL_2^\mu \right)_j \overline{f_j}  - \left(\opL_1^{\mu^*} \right)_j f_j \Vert_{2} \leq \eta_j,
    $$
    where $\eta_j$ can be expected large.
\end{itemize}
Another reason to focus on the first-order part is the smoothness properties given in Theorem \ref{thm:smoothness_props}. Since the $\opL_2$ is a smoother FIO than $\opL_1$, it tends to \textit{spread} the features of $f$ and therefore the first-order part is \textit{richer} for encoding $f$. A way to emphasize the smoothness scale, it is possible to add to the inverse problem a discretized differential operator (finite difference for example), $\opP: \R^{P\times K} \to \R^{P\times K}$ acting on the energy variable, leading to solve
    $$
    \opP \left(\opL_1^\mu \right)_j f_j
    = \opP  \mathbf{Spec}_j \ \text{with} \ f\in\X_j \ \text{for some} \  j\in \N.
    $$
and to the model uncertainty
    $$
    \Vert \opP \left(\left(\opL_1^\mu \right)_j f_j   + \left(\opL_2^\mu \right)_j \overline{f_j}  - \left(\opL_1^{\mu^*} \right)_j f_j \right)\Vert_{2} \leq \eta_j^\opP.
    $$
We observe empirically that the use of a differential operator reduces the model uncertainty, \textit{i.e.} $\eta_j^\opP < \eta_j$, but this remains to be proved. This strategy was successfully applied: in 3D using FBP-type algorithm for the extraction of the contours in \cite{Rigaud21} and in 2D using iterative total-variation (TV) regularization in \cite{KugerRigaud21,kuger2020joint}. 
However errors and artifacts due to an inaccurate model can appear and need to be addressed using data-driven algorithm.
\\

\par 
In \cite{Blanke_20} the authors consider a CT problem affected by motion of the patient, which leads to an inexact forward operator, as the motion is not explicitly known. They proposed to apply the regularized sequential subspace optimization (RESESOP) for solving inverse problem subject to model uncertainty and studied how the method is well-posed and regularizing for the exact inverse problem. In Section \ref{sec:resesop}, we propose to adapt this strategy for solving the semi-discrete and fully-discrete problems associated to CST. We also prove that the fully-discrete RESESOP is a regularization method for the semi-discrete SESOP solution. \\

\par A second approach consists in implementing the \textit{deep image prior} (DIP) unsupervised learning method in order to address our inexact inverse problem. Discussed in Section \ref{sec:DIP} and in Section \ref{sec:simu}, the standard loss function does not succeed to compensate for the model uncertainty in CST. Inspired from the RESESOP approach, we propose to adapt the loss function by incorporating the model uncertainty in the loss function which leads to similar results with the RESESOP method.

\section{Study and construction of a RESESOP algorithm for CST} \label{sec:resesop}

As discussed in Section \ref{sec:forward}, we are facing the issue of solving an inverse problem without explicitly knowing the forward operator. In order to get a valid reconstruction, we need to take the model uncertainty between the exact and inexact forward operator into account. In CST, the model uncertainty between $\opL_1^\mu$ and $\opL_1^{\mu^*}$ highly depends on the different source and detector positions, but also on the energy of the scattered photons. This is why we want to consider a system of inverse problems, instead of a single one.
To handle the model uncertainty issue for multiple inverse problems, we use the RESESOP-Kaczmarz procedure presented in \cite{Blanke_20}.
In the first part of this section, we give a recap on the functionality and regularization properties of this method.
Since we want to solve the fully discrete problem (\ref{eq:IP_fd}), the question arises whether the RESESOP outcome for the fully discrete problem regularizes the semi-discrete problem regarding $\opL_1^\mu$. This is inspected in the second part of this section and further, whether these reconstructions are stable with respect to the chosen subspace.
Last but not least, we explain how the RESESOP framework can be applied to CST.

\subsection{RESESOP-Kaczmarz for a system of linear inverse problems}\label{sec:RESESOP_recall}

Consider finitely many linear bounded operators $\opA_k: (\X, \Vert \cdot \Vert_\X) \longrightarrow (\Y_{k}, \Vert \cdot \Vert_{\Y_k})$
between Hilbert spaces $\X, \Y_k$, where $k\in \{ 0, 1,2, ..., K-1\}$ for some $K\in \mathbb{N}$.
We assume that only approximate versions of $\opA_k$ in the form of other linear bounded operators $ \opA_k^\eta: \X \longrightarrow \Y_{k}$ are available, satisfying
\begin{equation}
    \Vert \opA^\eta_k - \opA_k \Vert_{\X\to\Y_k} \leq \eta_k,
\end{equation}
where $\Vert \cdot \Vert_{\X\to\Y_k}$  denotes the operator norm of linear bounded functions between $\X$ and $\Y_k$.
In what follows, we abbreviate all norms by $\Vert \cdot \Vert$  when there is no ambiguity.
For the sake of readibility, we avoid writing $\eta_k$ in the superscript of the inexact forward operators. 
Further, the following notation will be useful:
$$
[n] := n \textup{ mod } K.
$$

The recap on the \emph{RESESOP-Kaczmarz} procedure, presented in \cite{Blanke_20}, will be twofold: First, in case that all $\opA_k$ are known and exact data $g_k$ in the range of $\opA_k$, noted $\textup{Ran}(\opA_k)$, are available, we illustrate the concept of SESOP-Kaczmarz for solving this system of inverse problems.
Second, we recall how this method can be extended if only inexact forward operators and noisy data is available.

Beforehand, we give an important definition.
\begin{definition}\label{def:hyperplanes_stripes}
Let $u\in \X$ and $\alpha\in \mathbb{R}$. Define the corresponding \emph{hyperplane} via
\begin{equation*}
    H(u, \alpha) := \big \{ x\in \X: \langle u, x\rangle = \alpha \big \}
,\end{equation*}
and \emph{upper halfspace} via
\begin{equation*}
    H_>(u, \alpha) := \big \{ x\in \X: \langle u, x\rangle \geq \alpha \big \}.
\end{equation*}
In addition, for $\xi\geq 0$ we define the corresponding  \emph{stripe} as
\begin{equation*}
    H(u, \alpha, \xi) := \big \{ x\in \X: \vert \langle u, x\rangle - \alpha \vert \leq \xi \big \}.
\end{equation*}
Note that $ H(u, \alpha, \xi) \subset  H(u, \alpha) $ for all $\xi$.
\end{definition}

\subsubsection{SESOP-Kaczmarz for exact forward operators and data}

\begin{figure}[t]
     \centering
     \begin{subfigure}[]
         \centering
         \includegraphics[scale=0.8]{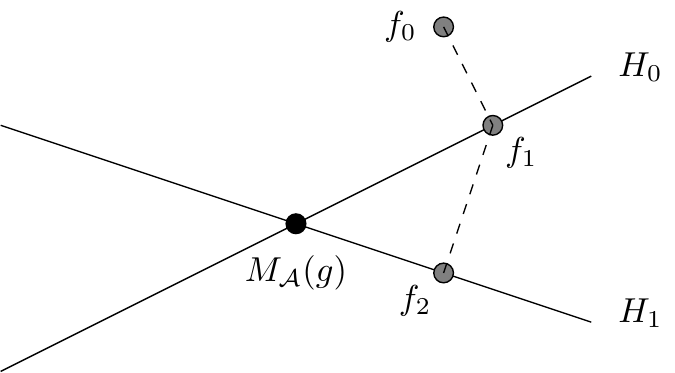}
     \end{subfigure}
     \begin{subfigure}[]
         \centering
         \includegraphics[scale=0.8]{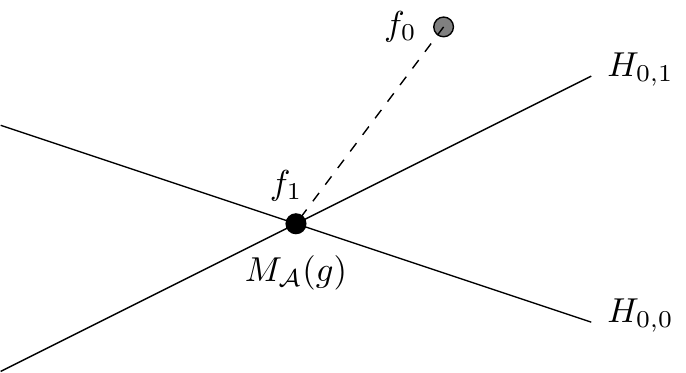}
     \end{subfigure}
     \caption{Iterative projection onto (a) hyperplanes and (b) intersection of hyperplanes}
    \label{fig:hyperplane_projection}
\end{figure}

We start with the observation that for any $w\in \Y_k$ the hyperplane
\begin{equation*}
    H(\opA_k^*w, \langle w, g_k\rangle) = \big \{ z \in \X: \langle w, \opA_k z\rangle = \langle w, g_k\rangle  \big \}
\end{equation*}
contains the solution set 
$$
M_{\opA}(g) := \big \{ z \in \X: \opA_j z = g_j \textup{ for all } j\big \}
$$
of the system of inverse problems.
Moreover, it follows from \emph{Riesz' Representation Theorem} \cite[Theorem 3.8.1]{Kreyzig} that 
\begin{equation*}
    \bigcap_{k=0}^{K-1} \bigcap_{w\in \Y_k} H(\opA_k^*w, \langle w, g_k\rangle) = M_\opA(g).
\end{equation*}
Therefore, the idea of Sequential Subspace Optimization is to choose $w_n\in \Y_{[n]}$ and iteratively project onto the corresponding hyperplanes $$H_n:= H(\opA_{[n]}^*w_n, \langle w_n, g_{[n]}\rangle).$$ That is, given a start iterate $f_0\in \X$, we set $f_n := \opPP_{H_n} (f_{n-1})$, where $\opPP_{H_n}$ denotes the orthogonal projection onto the closed convex set $H_n$, see Figure \ref{fig:hyperplane_projection}. 
The projection onto a single hyperplane can be computed by the following formula
\begin{lemma}\label{lem:projection_hyperplane}
Let $u\in X\setminus\{0\}$ and $\alpha\in \mathbb{R}$. Then the projection onto $H(u, \alpha)$ can be computed via
\begin{equation*}
    \opPP_{H(u, \alpha)}x = x - \frac{\langle u, x\rangle - \alpha}{\Vert u\Vert^2}u
\end{equation*}
for any $x\in \X$.
\end{lemma}

Instead of projecting onto a single hyperplane at each iteration, projecting onto the intersection of multiple hyperplanes may significantly increase the convergence rate, see \cite{Schoepfer_08}, and leads to multiple search directions, see \cite{Blanke_20}. This effect is also illustrated in Figure \ref{fig:hyperplane_projection}. However, we did not observe empirically a significant benefit in the convergence rate. This is the reason why we consider below only one search direction.

\subsubsection{RESESOP-Kaczmarz for inexact forward operator}
Let us now assume that only noisy data $g^\delta_{k}$ with noise levels $\delta_k \geq \Vert g_{k} - g^\delta_{k}\Vert$ and inexact forward operators $\opA^\eta_{k}$ are available, for $k\in \{0,1,...K-1\}$. At this point, we also set some convenient notation
 \begin{eqnarray*}
 \eta := (\eta_0,...,\eta_{K-1})^T \in \mathbb{R}^K, \\
 \delta := (\delta_0,...,\delta_{K-1})^T \in \mathbb{R}^K.
 \end{eqnarray*}
We further make the assumption that for some constant $\rho >0$ the restricted solution set 
\begin{equation*}
    M^\rho_\opA(g) := M_\opA(g) \cap B_\rho(0)
\end{equation*}
is non-empty, that means there is a solution whose norm is smaller than $\rho$.\\

\par The main issue is that for $w\in \Y_k$ the preceding hyperplanes $H((\opA_{[n]}^\eta)^*w, \langle w, g^\delta_{[n]} \rangle)$ may no longer contain the restricted solution set of the respective subproblem and hence, neither $M^\rho_\opA(g)$. This problem is tackled by projecting onto stripes instead of hyperplanes whose thickness is chosen in accordance with the level of noise $\delta$ and the model inexactness $\eta$, see Figure  \ref{fig:hyperplane_to_stripe}. This approach combined with the discrepancy principle is summarized in the following algorithm.

\begin{figure}
    \centering
    \includegraphics[scale=0.95]{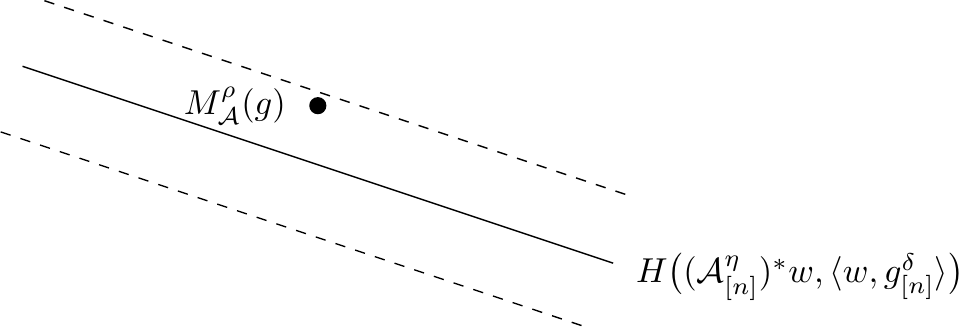}
    \caption{Increasing the thickness of the hyperplane so that the resulting stripe contains the restricted solution set}
    \label{fig:hyperplane_to_stripe}
\end{figure}

\begin{algorithm}[Similar to Algorithm 2.7. in \cite{Blanke_20}]\label{algo:resesop}
    Choose an initial value $f_0 := f_0^{\eta, \delta}\in B_\rho(0)\subset \X$ and a constant $\tau > 1$. 
    If the current iterate $f_n^{\eta, \delta}$ fulfills the discrepancy principle for the current subproblem, i.e.
    \begin{equation}\label{eq:discrepancy_principle}
        \Vert \opA^\eta_{[n]} f_n^{\eta, \delta} -g^\delta_{[n]} \Vert \leq \tau ( \rho\eta_{[n]} + \delta_{[n]}),
    \end{equation}
    set $f_{n+1}^{\eta, \delta} := f_n^{\eta, \delta}$. Otherwise set $w_{n}^{\eta, \delta} := \opA^\eta_{[n]} f_n^{\eta, \delta} - g^\delta_{[n]}$ and compute the next iterate via
    \begin{equation*}
        f_{n+1}^{\eta, \delta} := \opPP_{H_n^{\eta,\delta}} f_n^{\eta, \delta},
    \end{equation*}
    i.e. by projecting onto the stripe 
    \begin{equation*}
        H_n^{\eta, \delta} := H(u_n^{\eta, \delta}, \alpha_n^{\eta, \delta}, \xi_n^{\eta, \delta}),
    \end{equation*}
    where
    \begin{eqnarray*}
        u_n^{\eta, \delta} :=  (\opA_{[n]}^\eta)^* w_n^{\eta, \delta},\\
        \alpha_n^{\eta, \delta} := \langle w_n^{\eta, \delta},  g^\delta_{[n]}\rangle, \\
        \xi_n^{\eta, \delta} := (\rho\eta_{[n]} + \delta_{[n]})  \Vert w_n^{\eta, \delta} \Vert. 
    \end{eqnarray*}
    Stop iterating as soon as $f_{n+k} = f_n$ for all $k\in \{0, ..., K-1\}$.
\end{algorithm}

Note that for $(\eta, \delta) = 0$, the previous algorithm is just the SESOP-Kaczmarz procedure from the previous section. In this case we will omit all superindices and write for example $f_n$ instead of $f_n^{\eta, \delta}$.
As shown in \cite{Blanke_20}, by the construction of the stripe $H_n^{\eta, \delta}$ , it contains the restricted solution set $M^\rho_{\opA}(g):= M_\opA(g)\cap B_\rho(0).$

Furthermore, it might happen that $\Vert u_n^{\eta, \delta} \Vert = 0$ although the discrepancy principle (\ref{eq:discrepancy_principle}) is not fulfilled. In that case the stripe $H_n^{\eta, \delta}$ may be the empty set and hence the iteration step not well-defined. This leads to the following definition: 
\begin{definition}
    We call a start iterate $f^{\eta, \delta}_0$ of Algorithm \ref{algo:resesop} to be \emph{feasible} if $u_n^{\eta, \delta} \neq 0$, whenever the discrepancy principle (\ref{eq:discrepancy_principle}) is not fullfilled.
\end{definition}
However, it is to be noted here that if $(\opA_{[n]}^\eta)^*$ is injective, or equivalently $\opA_{[n]}^\eta$ is surjective, all start iterates are feasible.\\

\par 
Regarding the computation of the projection onto the stripe $H(u_n^{\eta, \delta}, \alpha^{\eta, \delta}_n, \xi^{\eta, \delta}_n)$ we have the following result.
\begin{lemma}\label{lem:projection_stripe}
If the iterate $f_n^{\eta, \delta}$ of Algorithm \ref{algo:resesop} does not fulfill the discrepancy principle (\ref{eq:discrepancy_principle}), then it holds that
\begin{equation*}\fl  
f^{\eta, \delta}_{n+1} = f^{\eta, \delta}_n - t^{\eta, \delta}_n u^{\eta, \delta}_n
\qquad
\text{with}
\qquad
    t^{\eta, \delta}_n := \frac{\langle u^{\eta, \delta}_n, f^{\eta, \delta}_n\rangle - (\alpha^{\eta, \delta}_n+\xi^{\eta, \delta}_n)}{\Vert u^{\eta, \delta}_n \Vert^2}.
\end{equation*}
\end{lemma}
\begin{proof}
Due to Proposition 3.5 in \cite{Blanke_20} we have $f^{\eta, \delta}_n \in H_>(u^{\eta, \delta}_n, \alpha^{\eta, \delta}_n+\xi^{\eta, \delta}_n)$. Hence, the claim follows from Lemma \ref{lem:projection_hyperplane}.
\end{proof}

\begin{remark}
Due to the form of $f^{\eta, \delta}_{n+1}$, the $u^{\eta, \delta}_n$ are also called \emph{search directions}.
\end{remark}

Next we state two theorems from \cite{Blanke_20}, which will be important in the next section.
The first one is about the convergence of the SESOP-Kaczmarz iteration.

\begin{theorem}{\cite[Theorem 3.3]{Blanke_20}}\label{thm:sesop_convergence}
    For $(\eta, \delta) = 0$ let $\{f_n\}_{n\in \mathbb{N}}$ be the sequence generated by Algorithm \ref{algo:resesop} for a feasible initial value $f_0\in \X$. If the parameters $t_n$ from Lemma \ref{lem:projection_stripe} are bounded, then it holds that
    \begin{equation*}
        \lim_{n\to \infty} f_n = \opPP_{M_\opA(g)} (f_0),
    \end{equation*}
    i.e. $(f_n)_{n\in \mathbb{N}}$ strongly converges to the projection of $f_0$ onto the solution set $M_{\opA}(g)$.
\end{theorem}

Before we state that the RESESOP-Kaczmarz method is indeed a regularization of the inverse problems $\opA_{k} f = g^\delta_{k},\ k\in\{0,\ldots,K-1\}$, it is to be noted here that whenever $(\eta, \delta)\neq0$ Algorithm \ref{algo:resesop} terminates after a finite number of iterations according to \cite[Lemma 3.7.]{Blanke_20}, that is 
\begin{equation}\fl
    n_*(\eta, \delta):= \min \Big \{ n\in \mathbb{N}: \Vert \opA^\eta_{[n']} f_{n'}^{\eta, \delta} - g_{[n']}^\delta \Vert \leq \tau(\rho\eta_{[n']} + \delta_{[n']}), \forall n' = n,...,n+K-1 \Big \}
\end{equation}
is finite. It is called \emph{finite stopping index}, but note that it is called auxilliary stopping index in \cite{Blanke_20}.

\begin{theorem}{\cite[Theorem 3.9.]{Blanke_20}}\label{thm:RESESOP_regularization}
Let $((\eta, \delta)_l)_{l\in \N}$ be a null-sequence
and $(\eta, \delta)_l\neq 0$ for all $l$. Given a feasible start iterate $f_0\in B_{\rho}(0)$, let $f_{n_*(\eta, \delta)_l}^{(\eta, \delta)_l}$ be the outcome of Algorithm \ref{algo:resesop}. Assume that the parameters $t_n^{(\eta, \delta)_l}$ from Lemma \ref{lem:projection_stripe} are bounded with respect to $n, l \in \mathbb{N}$. Then it holds that
\begin{equation*}
\lim_{l\to\infty}    f^{(\eta, \delta)_l}_{n_*(\eta, \delta)_l} = P_{M_\opA^\rho(g)}(f_0).
\end{equation*}
\end{theorem}

\subsection{Restriction of the domain space}
As mentioned in Section 2 we aim for an approximate solution of a semi-discrete inverse problem by considering a fully-discrete version of it. Therefore, in this section we assume $\Y$ to be finite dimensional and inspect what happens if we restrict the domain space $\X$ to some closed subspace $\X_j$.
To simplify the notation, we consider only a single forward operator $\opA, \opA^\eta:X\to \Y$ in the following and denote the exact data by $g\in \textup{Ran}(\opA)$. However, note that together with ideas from \cite{Blanke_20} it might be possible to adapt the results in this section to the case of multiple forward operators. 
The restrictions of $\opA$ and $\opA^\eta$ to $\X_j$ are denoted by
\begin{equation*}
    \opA^j, \opA^{\eta,j}: \X_j \longrightarrow \Y,
\end{equation*}
respectively. There restrictions are also linear and bounded operators between Hilbert spaces $\X_j$ and $\Y$.\\

\par 
In the first part of this section, we want to apply the RESESOP-Kaczmarz method to the restricted operators and use the preceding theory to observe in Corollary \ref{coro:regularization_fixed_j} that for fixed subspace $\X_j$, this yields under some assumption to a regularized solution of the semi-discrete inverse problem $\opA f = g^\delta$. 
In the second part, we prove stability with respect to the chosen subspace in Theorem \ref{thm:full_semi_regularization}.

\subsubsection{RESESOP-Kaczmarz applied to a restricted forward operator}\label{sec:RESESOP_restricted}
Throughout this section we make the assumption that

\begin{equation}\label{eq:g_in_range_Aj}
     g\in \textup{Ran}(\opA^j),
\end{equation}
which seems to be restrictive at first, but as $\Y$ is finite dimensional, the restricted operator $\opA^j$ has even a high chance of being surjective if the dimension of $\X_j$ is sufficiently large.
The assumption (\ref{eq:g_in_range_Aj}) implies that there exists some $\rho_j>0$ such that
$$ M^{\rho_j}_{\opA^j}(g) := M_{\opA^j}(g) \cap B_{\rho_j}(0) \neq \emptyset, $$
i.e. the (restricted) solution set is non-empty.

For start iterates $f_0^j := f_0^{\eta, \delta, j}\in \X_j$, we apply the RESESOP Algorithm \ref{algo:resesop} to the operator $\opA^{\eta, j}$ and extend the notation in Algorithm \ref{algo:resesop} by an additional superindex $j$, that is, we denote the iterates by $f_n^{\eta, \delta, j}$
and further set
\begin{eqnarray*}
    w_n^{\eta, \delta, j} = \opA^{\eta,j} f^{\eta, \delta, j}_n - g^\delta,\\
    u_n^{\eta, \delta, j} = (\opA^{\eta,j})^*u_n^{\eta, \delta, j},\\
    \alpha_n^{\eta, \delta, j} = \langle w_n^{\eta, \delta, j}, g^\delta \rangle,\\
    \xi_n^{\eta, \delta, j} = \Vert w_n^{\eta, \delta, j}\Vert (\eta \rho + \delta).
\end{eqnarray*}
Also, we replace $\rho$ by $\rho_j$ so that the stripes $H_n^{\eta, \delta, j}$ will contain the restricted solution set $M^{\rho_j}_{\opA^j}(g)$.  Again, if $(\eta, \delta) = 0$, we omit them in the superindex, for example, we then write $f_n^j$ instead of $f_n^{0,0,j}$. The theory presented in Section \ref{sec:RESESOP_recall} is also applicable to the restricted operators $\opA^j$ and $\opA^{\eta, j}$, which also means that for $(\eta, \delta)\neq 0$ there is a finite stopping index
\begin{equation}\label{eq:stopping_index}
    n_*(\eta, \delta, j) := \min \Big \{ n: \Vert \opA^{\eta,j} f^{\eta, \delta, j}_n - g^\delta\Vert \leq \tau(\rho_j\eta + \delta)\Big \}\in \mathbb{N}.
\end{equation}
and for $n\geq n_*(\eta, \delta, j)$ it holds that $f^{\eta, \delta, j}_n = f^{\eta, \delta, j}_{n_*(\eta, \delta, j)}$.
Moreover, from Theorem \ref{thm:sesop_convergence} and Theorem \ref{thm:RESESOP_regularization} we immediately obtain the following result: 

\begin{corollary}
    For $(\eta, \delta) = 0$ let $(f_n^j)_{n\in \mathbb{N}}$ be the sequence generated by Algorithm \ref{algo:resesop} for a feasible initial value $f_0^j\in \X$. If the parameters $t_n^j$ from Lemma \ref{lem:projection_stripe} are bounded, then it holds that
    \begin{equation*}
        \lim_{n\to \infty} f_n^j = \opPP_{M_{\opA^j}(g)} (f_0^j)\in  M_{\opA}(g),
    \end{equation*}
    i.e. $(f_n^j)_{n\in \mathbb{N}}$ strongly converges to the projection of $f_0^j$ onto the solution set $M_{\opA^j}(g)$.
\end{corollary}
This means that in case of $(\eta, \delta)=0$, the application of the SESOP algorithm to the resctricted operator $\opA^j$ indeed converges to a solution of $\opA f = g$.
 Further, in case of model uncertainty or noisy data, applying RESESOP to  $\opA^{\eta,j}$ regularizes the inverse problem regarding $\opA$ in the following way.
\begin{corollary}\label{coro:regularization_fixed_j}
    Let $((\eta, \delta)_l)_{l\in\N}$ be a null sequence and $(\eta, \delta)_l\neq 0$ for all $l$. Further, for some feasible start iterate $f^j_0\in B_{\rho_j}(0)$ let $f_{n_*(\eta_l, \delta_l,j)}^{\eta_l, \delta_l,j}$ be the outcome of Algorithm \ref{algo:resesop}. Assume that the parameters $t_n^{\eta_l, \delta_l,j}$ from Lemma \ref{lem:projection_stripe} are bounded with respect to $n, l \in \mathbb{N}$. Then it holds that
    \begin{equation*}
    \lim_{l\to\infty}    f^{\eta_l, \delta_l, j}_{n_*(\eta_l, \delta_l, j)} = P_{M^{\rho_j}_{\opA^{j}}(g)} (f^j_0)\in  M^{\rho_j}_{\opA}(g).
    \end{equation*}
\end{corollary}
However, at this point it is not clear, whether the RESESOP reconstruction is stable with respect to the chosen subspace $\X_j$. This is analyzed in the next subsection. \\

\par 
We recall now a descent property for the RESESOP iterates from \cite{Blanke_20}, which will be helpful for the analysis in the next subsection.
\begin{proposition}{\cite[Proposition 3.5.]{Blanke_20}}\label{prop:descent_property}
    If $\Vert \opA^{\eta,j} f^{\eta, \delta, j}_n - g^\delta \Vert > \tau (\rho_j \eta + \delta)$, then it holds that
    \begin{itemize}
        \item[a)] $f_n^{\eta, \delta, j}$ is contained in the half-space $ H_>(u_n^{\eta, \delta, j}, \alpha_n^{\eta, \delta, j} + \xi_n^{\eta, \delta, j})$.
        \item[b)] For all $z\in H(u_n^{\eta, \delta, j}, \alpha_n^{\eta, \delta, j}, \xi_n^{\eta, \delta, j})$ it holds that $$\Vert z - f_{n+1}^{\eta, \delta, j} \Vert^2 \leq \Vert z - f_{n}^{\eta, \delta, j}\Vert^2 -\left( \frac{\Vert w_n^{\eta, \delta, j} \Vert \big(\Vert w_n^{\eta, \delta, j} \Vert - (\rho\eta + \delta) \big) }{\Vert u_n^{\eta, \delta, j} \Vert}\right)^2.$$
    \end{itemize}
    In particular, b) holds for all elements $z$ of the restricted solution set $M^{\rho_j}_{\opA^j}(g)$.
\end{proposition}
The following Lemma addresses the computation of the adjoint of the restricted forward operator.

\begin{lemma}\label{lem:adjoint_restriction}
For all $y\in \Y$ it holds that $(\opA^j)^*y = \opPP_{\X_j} \opA^*y$ and similarly for $(\opA^{\eta,j})^*$.
\end{lemma}
\begin{proof}
Let $y\in \Y$. For all $v \in \X_j$ it holds that
\begin{eqnarray*}
    \langle (\opA^j)^*y - \opA^* y, v\rangle = \langle y, \opA^j v - \opA v\rangle = \langle y, \opA v - \opA v\rangle = 0.
\end{eqnarray*}
Therefore, $(\opA^j)^*y - \opA ^*y$ is orthogonal to $\X_j$. As $(\opA^j)^*y \in \X_j$, we conclude
\begin{equation*}
    0 = \opPP_{X_j} ( (\opA_j)^*y - \opA^*y ) = \opPP_{X_j} (\opA^*y) -  (\opA^j)^*y.
\end{equation*}
\end{proof}

\subsubsection{Stability with respect to the chosen subspace}\label{sec:RESESOP_restricted_stability}
In this section we consider a nested sequence of closed subspaces $\X_j$ of $\X$, i.e. \begin{equation}\label{assum:nestedness}
    \X_j \subset \X_{j+1} \textup{ for all } j\in \mathbb{N}.
\end{equation}
Further, we assume that
\begin{eqnarray}\label{assum:contain_Nperp} 
    N(\opA)^\perp \subseteq \overline{\bigcup_j \X_j}
\end{eqnarray}
and make the stronger assumption, compared to section \ref{sec:RESESOP_restricted}, that there exists some $J\in \mathbb{N}$ such that the restriction $\opA^J$ of $\opA$ to $\X_J$ is surjective. Due to the nestedness (\ref{assum:nestedness}), it follows that $(\opA^j)$ is surjective for $j\geq J.$ Therefore, without loss of generality we assume all $\opA^j$ to be surjective. Furthermore, it is to be noted here that for closed subspaces $V\subseteq \X_j$ the expression $V^\perp$ stands for the orthogonal complement in $\X$. The orthogonal complement in $\X_j$ is denoted by $V^{\perp_j}$.\\

\par 
The main goal of this section is to prove the following result:

\begin{theorem}\label{thm:full_semi_regularization}
    Let $f_0^j:= f_0^{\eta, \delta, j}\in N(\opA^j)^{\perp_j}\cap B_{\rho_j}(0)$, $j \in\N$, be start iterates for the RESESOP method applied to $\opA^{\eta,j}$ and assume that $f_0^j$ converges to some start iterate $f_0\in N(\opA)^\perp\cap B_{\rho'}(0)$ for the SESOP method applied to $\opA$. Given some sequence $((\eta, \delta, j)_l)_{l\in N}$ converging to $(0,0,\infty)$ we assume that the parameters $t_n^{(\eta, \delta,j)_l}$ and $t_n$ from Lemma \ref{lem:projection_stripe} are bounded with respect to $n, l \in \mathbb{N}$. Then it holds that
    \begin{equation*}
        \lim_{l\to\infty} f_{n_*(l)}^{(\eta, \delta, j)_l} = \opPP_{M_\opA(g)}f_0,
    \end{equation*}
    where $n_*(l) := n_*((\eta, \delta, j)_l)$ is the finite stopping index from (\ref{eq:stopping_index}).
\end{theorem}
\begin{remark}
     We say that a sequence $(\eta, \delta, j)_l$ is convergent to $(0,0,\infty)$, if for all $\varepsilon>0$ and $N>0$ there exists some $L\in \mathbb{N}$ such that for all $l\geq L$
     \begin{equation*}
         \vert \eta_l\vert < \varepsilon, \vert \delta_l\vert < \varepsilon \textup{ and } j_l > N.
     \end{equation*}
     Moreover, we want to emphasize again that the assumption of $\opA^j$ being surjective implies that all start iterates $f_0^j:= f_0^{\eta, \delta, j}\in \X_j$ and $f_0\in \X$, respectively, are feasible. Therefore, we omitted these conditions in Theorem \ref{thm:full_semi_regularization}.
     \end{remark}

In order to prove this theorem, some preparations are required. First, we inspect the projections onto solution sets.

\begin{lemma}\label{lem:generalized_inverse}
    Let $f_0 \in N(\opA)^\perp$. Then it holds that 
    \begin{equation*}
        \opPP_{M_\opA(g)} f_0 = \opA^+ g,
    \end{equation*}
    where $\opA^+$ denotes the \emph{generalized inverse} of $\opA$, see for example \cite[Definition 2.1.5.]{Rieder}.
    By analogy, it holds that 
    \begin{equation*}
        \opPP_{M_{\opA^j}(g)} f_0^j = (\opA^j)^+ g,
    \end{equation*}
    for any $f_0^j\in N(\opA^j)^{\perp}.$
\end{lemma}
\begin{proof}
    Let $f \in M_\opA(g)$. As the solution set is an affine set, namely $M_\opA(g) = f + N(\opA)$, the corresponding orthogonal projection can be computed via
    \begin{eqnarray*}
        \opPP_{M_\opA(g)} f_0 = f + \opPP_{N(\opA)}(f_0 - f).
    \end{eqnarray*}
    Due to $f_0\in N(\opA)^\perp$ we conclude
    \begin{eqnarray*}
        \Vert \opPP_{M_\opA(g)} f_0 \Vert = \Vert f - \opPP_{N(\opA)}f\Vert= \Vert \opPP_{N(\opA)^\perp}f  \Vert \leq \Vert f \Vert.
    \end{eqnarray*}
    This means that $\opPP_{M_\opA(g)} f_0$ is the minimum-norm solution of $\opA f = g$.
\end{proof}
At this point it is to be noted here that, as $\opA$ and $\opA^j$ both map into a finite dimensional space $\Y$, their generalized inverses are bounded operators defined on the whole space $\Y$, see \cite[Satz 2.1.8]{Rieder}.
The next step is to show that $(\opA^j)^+g$ converges to $\opA^+g$. For that purpose we need some results on orthogonal projections:

\begin{lemma}
\label{lem:projections_convergence}
Let $(V_j)_{j\in \mathbb{N}}$ be a sequence of nested subspaces of $\X$. For all $f\in \X$ it holds that
\begin{equation*}
    \lim_{j\to \infty} \opPP_{V_j}f = \opPP_{V}f,
\qquad
\text{where}
\qquad
    V := \overline{\bigcup_{j\in \mathbb{N}} V_j}.
\end{equation*}
In particular, for all $f\in \X$ it holds that
\begin{equation*}\fl 
    \lim_{j\to \infty} \opPP_{\X_j}f = \opPP_{\overline{\bigcup_j \X_j}}f
\qquad
\text{as well as}
\qquad
    \lim_{j\to \infty} \opPP_{N(\opA^j)}f = \opPP_{\overline{\bigcup_j N(\opA^j)}}f.
\end{equation*}
\end{lemma}
\begin{proof}
Let $f\in V$ and $\varepsilon>0$. By definition of $V$, there exists some $N\in \mathbb{N}$ and $f_\varepsilon\in V_N$ such that $\Vert f- f_\varepsilon\Vert<\varepsilon$. Due to the nestedness of the $V_j$ we conclude that $f_\varepsilon$ belongs to all $V_j$ for $n\geq N$.
Thus, 
\begin{eqnarray*}
    \Vert \opPP_{V}f - \opPP_{V_j}f\Vert &\leq \Vert f-f_\varepsilon \Vert + \Vert \opPP_{V_j}f_\varepsilon - \opPP_{V_j} f \Vert \\
    &\leq (1+ \Vert \opPP_{V_j}\Vert )\varepsilon \leq 2\varepsilon.
\end{eqnarray*}
Therefore $\opPP_{V_j}f$ converges to $\opPP_Vf$. Due to the nestedness, the  general case $f\in \X$ follows analogously by rewriting $f=f_V + f_{V^\perp}$ for $f_V \in V$ and $f_{V^\perp}\in V^\perp.$
\end{proof}

\begin{corollary}\label{coro:proj_sol_space_conv_closure}
    Let $(x_n)_{n\in \N}$ be a sequence in $\X$ converging to some $x\in \X$. It holds that
    \begin{equation*}
        \lim_{j\to \infty} \opPP_{M_{\opA^j}(g)}x_j = \opPP_{V}x,
    \qquad    \text{where}\qquad        
    V := \overline{\bigcup_j M_{A^j}(g)}.
    \end{equation*}
\end{corollary}
\begin{proof}
As $g\in \textup{Ran}(\opA^j)$ for all $j\in \N$ and by the nestedness of the $\X_j$, there is some $x'\in \X$ with $x'\in M_{\opA^j}(g)$ for all $j\in \mathbb{N}.$ Therefore, we can write 
\begin{equation*}
    M_{\opA^j}(g) = x' + N(\opA^j).
\end{equation*}
Hence, we conclude by Lemma \ref{lem:projections_convergence} that
\begin{eqnarray*}
    \lim_{j\to\infty} \opPP_{M_{\opA^j}(g)} x &= x' + \lim_{j\to\infty} \opPP_{N\opA^j)} (x'-x)\\
    &= x' + \opPP_{\overline{\bigcup_j N(\opA^j)}} (x'-x)\\
    &=\opPP_{x' + \overline{\bigcup_j N(\opA^j)}} x
    =\opPP_{\overline{x' + \bigcup_j N(\opA^j)}} x
    =\opPP_V x.
\end{eqnarray*}
Therefore, for $\varepsilon>0$ there exists a $J\in \mathbb{N}$ such that
\begin{equation*}
    \Vert \opPP_{M_{\opA^j}(g)} x - \opPP_V x \Vert \leq \frac{\varepsilon}{2} \qquad \forall \; j\geq J
\end{equation*}
As $(x_n)_n$ converges to $x$ there is some $J'\geq J$ such that $\Vert x_j - x\Vert \leq \frac{\varepsilon}{2}$ for all $j\geq J'$. Altogether, for $j\geq J'$ it follows

\begin{eqnarray*}\fl
    \Vert \opPP_V x - \opPP_{M_{\opA^j}(g)} x_j \Vert &= \Vert \opPP_V x - \opPP_{M_{\opA^j}(g)} x_j \Vert \\
    &\leq \Vert \opPP_V x - \opPP_{M_{\opA^j}(g)} x \Vert + \Vert \opPP_{M_{\opA^j}(g)} x - \opPP_{M_{\opA^j}(g)} x_j \Vert \\
    &\leq \varepsilon/2 + \Vert \opPP_{M_{\opA^j}(g)} \Vert \cdot \Vert x - x_j\Vert \leq \varepsilon
\end{eqnarray*}
which ends the proof.
\end{proof}

 Lemma \ref{lem:projections_convergence} enables us to derive another important result, which is a special case of \cite[Lemma 6.1.5.]{Rieder}.
\begin{lemma}\label{lem:projection_to_zero}
For all $f \in N(\opA)^\perp$ it holds that $\opPP_{N(\opA^j)} f \overset{j\to\infty}{\longrightarrow} 0$.
\end{lemma}
\begin{proof}
    Let $\varepsilon>0$ and set $V:= \overline{\bigcup_j \X_j}$. As $\opA^*$ maps $\textup{Ran}(\opA)$ into a dense subset of $N(\opA)^\perp$, there exists an element $y\in \textup{Ran}(\opA)$ such that $\Vert x- \opA^*y\Vert < \varepsilon/2$. It follows from Lemma \ref{lem:projections_convergence} that
    \begin{equation*}
        \opPP_{\X_j} \opA^*y \overset{j\to \infty}{\longrightarrow} \opPP_V \opA^*y = \opA^*y,
    \end{equation*}
    as $\opA^*y\in N(\opA)^\perp\subseteq V$ by assumption (\ref{assum:contain_Nperp}). Therefore, there exists $J\in \mathbb{N}$ such that $\Vert \opPP_{\X_j} \opA^*y - \opA^*y   \Vert \leq \varepsilon/2$ for all $j\geq J$. Using Lemma \ref{lem:adjoint_restriction}, the desired result follows since for all $j \geq J$ we have
    \begin{eqnarray*}\fl
        \Vert \opPP_{N(\opA^j)}f \Vert &\leq \Vert  \opPP_{N(\opA^j)}(f-\opA^*y) \Vert + \Vert \opPP_{N(\opA^j)} (\opA^j)^*y \Vert + \Vert \opPP_{N(\opA^j)}(\opA^* - (\opA^j)^*)y \Vert\\
        &\leq \varepsilon/2 + \Vert \opA^*y - (\opA^j)^*y\Vert = \varepsilon/2 + \Vert \opA^*y - \opPP_{\X_j}\opA^*y\Vert \leq \varepsilon.
    \end{eqnarray*}
\end{proof}
Now we are able to prove the convergence of the sequence of the generalized solutions $(\opA^j)^+ g$. Regarding the proof, we follow the ideas in  \cite[Section 6.1.2.]{Rieder}.
\begin{theorem}\label{thm:generalized_inverses_converge}
    It holds that $((\opA^j)^+ g)_j$ converges to $\opA^+ g$. 
\end{theorem}
\begin{proof}
We set $f^+:= \opA^+ g$ and $f^+_j:= (\opA^j)^+ g$. 
We recall the notation $N(\opA^j)^{\perp_j}$ for the orthogonal complement of $N(\opA^j)$ in $\X^j$. 
We start with the consideration
\begin{eqnarray*}
    f^+ - f^+_j &= f^+ - (\opA^j)^+g\\
    &= (I - (\opA^j)^+\opA) f^+\\
    &= (I - (\opA^j)^+\opA) (f^+- \opPP_{\X_j}f^+ ) - (I - (\opA^j)^+\opA) \opPP_{\X_j}f^+,
\end{eqnarray*}
where $I$ denotes the identity on $\X$.
Regarding the second term we apply \cite[Theorem 2.1.9]{Rieder} to derive
\begin{eqnarray*}
    (I - (\opA^j)^+\opA) \opPP_{\X_j}f^+ &= \opPP_{\X_j}f^+ - (\opA^j)^+\opA \opPP_{\X_j}f^+ \\
    &= \opPP_{\X_j}f^+ - (\opA^j)^+\opA^j \opPP_{\X_j}f^+\\
    &= \opPP_{\X_j}f^+ - \opPP_{N(\opA^j)^{\perp_j}} \opPP_{\X_j}f^+\\
    &=\opPP_{\X_j}f^+ - \opPP_{N(\opA^j)^{\perp_j}} f^+ = \opPP_{N(\opA^j)} f^+,
\end{eqnarray*}
which converges to 0, according to Lemma \ref{lem:projection_to_zero}, as $f^+\in N(\opA)^\perp.$

It remains to show that also the first term above converges to zero.
We start by mentioning that the generalized solutions of $\opA^j$ and $\opA$ are bounded operators. By the surjectivity of the $\opA^j$ we conclude for arbitrary $y\in \Y$ that $(\opA^{j})^+ y$ solves the inverse problem $\opA^j f = y$. It is due to the nestedness (\ref{assum:nestedness}) of the $\X^j$ that it solves also $\opA^{j+1} f = y$. Thus, by definition of the generalized inverse we obtain
\begin{equation*}
    \Vert (\opA^{j+1})^+y \Vert \leq \Vert (\opA^{j})^+y \Vert \textup{ for all } j\in \mathbb{N}.
\end{equation*}
As a consequence, there is a $C>0$ such that 
\begin{equation}\label{eq:thm:generalized_inverse_convergence}
    \Vert (\opA^j)^+\Vert \leq C \textup{ for all } j\in \mathbb{N}.
\end{equation}
Therefore, we are able to estimate
\begin{eqnarray*}
    \Vert (I - (\opA^j)^+\opA) (f^+- \opPP_{\X_j}f^+ ) \Vert &\overset{(\ref{eq:thm:generalized_inverse_convergence})}{\leq}(1+C) \Vert \opA \Vert \cdot \Vert f^+- \opPP_{\X_j}f^+ \Vert.
\end{eqnarray*}
As $f^+\in N(\opA)^\perp \subseteq \overline{\bigcup_j \X_j}$, the right-hand-side of the inequality converges to 0, according to Lemma \ref{lem:projections_convergence}. Finally, we have shown that $(f^+_j)_j$ indeed converges to $f^+.$
\end{proof}
\begin{remark}
If we had not assumed the surjectivity of $\opA^j$, then the estimate $\Vert (\opA^{j+1})^+ y\Vert\leq \Vert (\opA^j)^+ y\Vert$ in the proof of the previous theorem might not hold in general.
\end{remark}
\begin{corollary}\label{coro:not_so_bad_condition}
    For any $f \in N(\opA)^\perp$ it holds that 
    \begin{equation*}
        \opPP_V f = \opPP_{M_{\opA}(g)} f,
    \qquad \text{where} \qquad
        V:= \overline{\bigcup_{j\in \mathbb{N}} M_{\opA^j}(g)}.
    \end{equation*}
\end{corollary}
\begin{proof}
According to Lemma \ref{lem:generalized_inverse} we have that $\opPP_{M_{\opA}(g)} f = \opA^+ g$.
Since $N(\opA)^\perp \subset N(\opA^j)^\perp$, it follows that $\opPP_{M_{\opA^j}(g)} f = (\opA^j)^+ g$.
Therefore, applying Corollary \ref{coro:proj_sol_space_conv_closure} and Lemma \ref{lem:generalized_inverse} and Theorem \ref{thm:generalized_inverses_converge} yields
\begin{eqnarray*}
    \opPP_V f = \lim_{j\to\infty} \opPP_{M_{\opA^j}(g)} f
    &=
    \lim_{j\to\infty}
    (\opA^j)^+ g = \opA^+ g = \opPP_{M_\opA(g)} f.
\end{eqnarray*}
\end{proof}

We recall our assumption that $\opA f = g$ has a solution in $B_\rho(0)$ and $\opA^j f = g$ has a solution in $B_{\rho_j}(0)\cap \X_j$. Note that due to the nestedness of the $\X_j$ we may assume that $\rho_{j+1} \leq \rho_j$ for all $j$. In all what follows, we consider for the application of RESESOP to $\opA^{\eta, j}$ and for the application of SESOP to $\opA$ a uniform
\begin{equation}
    \rho' := \max_j \{\rho, \rho_j\} = \max \{ \rho, \rho_1\}.
\end{equation}
This guarantees that the restricted solution sets $M_{\opA^j}(g) \cap B_{\rho'}(0)$ are non-empty for all $j$ and are contained in the respective stripes $H_n^{\eta, \delta, j}$ from Algorithm \ref{algo:resesop}.

In order to prove Theorem \ref{thm:full_semi_regularization}, we follow the same strategy as in \cite{Blanke_20} for proving Theorem \ref{thm:RESESOP_regularization} for the case $\X^j = \X$.
For that, we first show that the iterates $f_n^{\eta, \delta, j}$ from the application of RESESOP to $\opA^{\eta, j}$ converges to $f_n$, the iterates from the application of SESOP to $\opA$. To simplify the notation, we denote by $x^{\eta, \delta, j} \to x$ the following
$$\lim_{(\eta,\delta ,j)\to(0,0,\infty)} x^{\eta, \delta, j} = x$$ 
and state a useful result.
\begin{lemma}\label{lem:induction_helper}
    Let $n\in \mathbb{N}$ and consider a sequence $(\eta, \delta, j)$ converging to $(0,0,\infty)$. If $f_n^{\eta, \delta, j}$ converges to $f_n$, then
    \begin{equation*}
         w_n^{\eta, \delta, j} \to w_n, \, \,   u_n^{\eta, \delta, j} \to u_n, \, \, \alpha_n^{\eta, \delta, j}\to \alpha_n, \, \, \xi_n^{\eta, \delta, j} \to \xi_n.
    \end{equation*}
\end{lemma}
\begin{proof}
First, note that for $x\in \X_j$ it holds that $\opA^jx = \opA x$ and $\opA^{\eta,j}x = \opA^\eta x$.
Consider
\begin{eqnarray*}
    \Vert w_n^{\eta, \delta, j} - w_n \Vert &= \Vert \opA^{\eta, j} f_n^{\eta, \delta, j} - g^\delta - (\opA f_n - g)\Vert\\
    &\leq \delta + \Vert \opA^\eta f_n^{\eta, \delta, j} - \opA f_n \Vert \\
    &\leq \delta + \Vert \opA^\eta f_n^{\eta, \delta, j} - \opA^\eta f_n \Vert + \Vert \opA^\eta f_n - \opA f_n \Vert \\ 
    &\leq \delta + \Vert \opA^\eta \Vert \cdot \Vert f_n^{\eta, \delta, j} - f_n \Vert + \eta \Vert f_n \Vert.
\end{eqnarray*}
Note that $\Vert \opA^\eta \Vert$ is bounded, as $\opA^\eta$ converges to $\opA$. Moreover, $f_n^{\eta, \delta, j}$ converges to $f_n$ by assumption, so that we conclude from the estimation above that $w_n^{\eta, \delta, j}$ converges to $w_n$ if $(\eta, \delta, j)$ converges to $(0,0,\infty)$.\\
Using Lemma \ref{lem:adjoint_restriction}, we obtain
\begin{eqnarray*}\fl 
    \Vert u_n^{\eta, \delta, j} - u_n \Vert &= \Vert (\opA^{\eta, j})^* w_n^{\eta, \delta, j} - \opA^*w_n \Vert \\
    &= \Vert \opPP_{\X_j}(\opA^\eta)^* w_n^{\eta, \delta, j} - \opA^*w_n \Vert \\
    &\leq \Vert \opPP_{\X_j}(\opA^\eta)^* w_n^{\eta, \delta, j} - \opPP_{\X_j}(\opA^\eta)^* w_n \Vert  
    + \Vert \opPP_{\X_j}(\opA^\eta)^* w_n - \opA^*w_n \Vert \\
    &\leq \Vert (\opA^\eta)^*\Vert \cdot \Vert w_n^{\eta, \delta, j}  - w_n \Vert 
    + \Vert \opPP_{\X_j}(\opA^\eta)^* w_n - \opA^*w_n \Vert.
\end{eqnarray*}
Therefore, it follows that
\begin{eqnarray*}\fl
    \Vert \opPP_{\X_j}(\opA^\eta)^* w_n - \opA^*w_n \Vert &\leq \Vert \opPP_{\X_j}(\opA^\eta)^* w_n - \opPP_{\X_j}\opA^* w_n \Vert
    + \Vert \opPP_{\X_j}\opA^* w_n - \opA^*w_n \Vert \\
    &\leq \Vert (\opA^\eta)^* - \opA^* \Vert\cdot \Vert w_n\Vert + \Vert \opPP_{\X_j}\opA^* w_n - \opA^*w_n \Vert\\
    &\leq \eta \Vert w_n \Vert + \Vert \opPP_{\X_j}\opA^* w_n - \opA^*w_n \Vert.
\end{eqnarray*}
Since $\opA^* w_n\in N(\opA)^\perp \subseteq \overline{\bigcup_j \X_j}$, the second term on the right-hand side converges to 0 by Lemma \ref{lem:projections_convergence}.
Therefore, it follows by the estimations above and the convergence of the $w_n^{\eta, \delta, j}$ that $u_n^{\eta, \delta, j}$ converges to $u_n$. It is now straightforward to see that also $\alpha_n^{\eta, \delta, j}$, $\xi_n^{\eta, \delta, j}$ converge to $\alpha_n, \xi_n$, respectively.
\end{proof}

With this result we are able to prove the following.
\begin{corollary}{\cite[Lemma 3.8.]{Blanke_20}}\label{coro:fully_semi_iterates_converge}
    Assume that the RESESOP start iterates $f_0^{\eta, \delta, j}\in B_{\rho'}(0)\cap \X_j$ converge to the SESOP start iterate $f_0\in N(\opA)^\perp\cap B_{\rho'}(0)$. Then it holds that
    $$f_n^{\eta, \delta, j} \to f_n \qquad \forall \ n\in \N.$$
\end{corollary}
\begin{proof} We prove this statement by induction. The base case for $n=0$ is fulfilled just by assumption. Assume that $f_n^{\eta, \delta, j}$ converges to $f_n$. From Algorithm \ref{algo:resesop} and Proposition \ref{prop:descent_property}, we observe that

\begin{eqnarray*}\fl
 f_{n+1}^{\eta, \delta, j}  = \Bigg\{ \begin{array}{ll}
         f_n^{\eta, \delta, j} & \mbox{if $\Vert w_n^{\eta, \delta, j}\Vert \leq \tau(\rho'\eta+\delta)$};\\
        f_n^{\eta, \delta, j} - \frac{\langle u_n^{\eta, \delta, j}, f_n^{\eta, \delta, j}\rangle - (\alpha_n^{\eta, \delta, j} + \xi_n^{\eta, \delta, j})}{\Vert u_n^{\eta, \delta, j}\Vert^2} u_n^{\eta, \delta, j} & \mbox{otherwise}.\end{array} 
\end{eqnarray*}

First, consider sequences in 
\begin{equation*}
    I_1 := \{ (\eta, \delta, j): \Vert w_n^{\eta, \delta, j} \Vert \leq  \tau(\rho'\eta + \delta) \}.
\end{equation*}
This means, that for those $(\eta, \delta, j)$ the discrepancy principle (\ref{eq:discrepancy_principle}) is fulfilled at iteration index $n$, which means $f_{n+1}^{\eta, \delta, j} = f_n^{\eta, \delta, j}$.
In this case, we conclude by the induction hypothesis and the previous Lemma \ref{lem:induction_helper}
\begin{equation*}
    \Vert w_n \Vert = \lim \Vert w_n^{\eta, \delta, j} \Vert = 0,
\end{equation*}
which implies $ \opA f_n - g = w_n = 0$ and hence $f_{n+1} = f_n = \lim f_n^{\eta, \delta, j} = \lim f_{n+1}^{\eta, \delta, j}.$

Second, consider sequences in 
\begin{equation*}
    I_2 := \{ (\eta, \delta, j): \Vert w_n^{\eta, \delta, j} \Vert > \tau(\rho'\eta + \delta) \}.
\end{equation*}
If $w_n\neq 0$, it follows by the feasibility of $f_0$ that $u_n\neq 0$. Therefore, we conclude by the induction hypothesis, Lemma \ref{lem:induction_helper} and Proposition \ref{prop:descent_property} a) that \begin{eqnarray*}
    f_{n+1}^{\eta, \delta, j} &= f_n^{\eta, \delta, j} - \frac{\langle u_n^{\eta, \delta, j}, f_n^{\eta, \delta, j}\rangle - (\alpha_n^{\eta, \delta, j} + \xi_n^{\eta, \delta, j})}{\Vert u_n^{\eta, \delta, j}\Vert^2} u_n^{\eta, \delta, j}\\
    &\to f_n - \frac{\langle u_n, f_n\rangle - (\alpha_n + \xi_n)}{\Vert u_n\Vert^2} u_n = f_{n+1}
\end{eqnarray*}
Let now $w_n = 0$, which means that $f_n$ is the outcome of the SESOP algorithm. Due to Theorem \ref{thm:sesop_convergence}, we conclude $f_{n+1} = f_n = \opPP_{M_\opA(g)} f_0$.
Let $\varepsilon >0$.
We insert $z:= \opPP_{M_{\opA^j}(g)} f_0^j$ into Proposition \ref{prop:descent_property} b) and obtain
\begin{eqnarray*}\fl
     \left\Vert\frac{\langle u_n^{\eta, \delta, j}, f_n^{\eta, \delta, j}\rangle - (\alpha_n^{\eta, \delta, j} + \xi_n^{\eta, \delta, j})}{\Vert u_n^{\eta, \delta, j}\Vert^2} u_n^{\eta, \delta, j} \right\Vert &=  \left( \frac{\Vert w_n^{\eta, \delta, j} \Vert \big(\Vert w_n^{\eta, \delta, j} \Vert - (\rho'\eta + \delta) \big) }{\Vert u_n^{\eta, \delta, j} \Vert}\right)^2\\
     &\leq \Vert \opPP_{M_{\opA^j}(g)} f_0^j  - f_n^{\eta, \delta, j} \Vert.  \\
     &\leq \Vert \opPP_{M_{\opA^j}(g)} f_0^j  - \opPP_{M_\opA(g)} f_0 \Vert + \Vert \opPP_{M_\opA(g)} f_0  - f_n^{\eta, \delta, j} \Vert.
\end{eqnarray*}
The first term on the right-hand side converges to zero according to Corollary \ref{coro:not_so_bad_condition} and Corollary \ref{coro:proj_sol_space_conv_closure}, whereas the second term converges to zero by the induction hypothesis.
Finally, this means that
\begin{eqnarray*}
    f_{n+1}^{\eta, \delta, j} &= f_n^{\eta, \delta, j} - \frac{\langle u_n^{\eta, \delta, j}, f_n^{\eta, \delta, j}\rangle - (\alpha_n^{\eta, \delta, j} + \xi_n^{\eta, \delta, j})}{\Vert u_n^{\eta, \delta, j}\Vert^2} u_n^{\eta, \delta, j} 
    \to f_n = f_{n+1}.
\end{eqnarray*}
Altogether we conclude that $f_n^{\eta, \delta, j}$ converges to $f_{n+1}$.
\end{proof}

Now we are able to prove the main result of this section.

\begin{proof}[Proof of Theorem \ref{thm:full_semi_regularization}]
\, 

First, each subsequence of $n_*(l)$ has a bounded or a monotonically increasing, unbounded subsequence. Therefore, it suffices to consider these two cases - namely $n_*(l)$ to be bounded or unbounded but monotonically increasing - and show that in both cases $f_{n_*(l)}^{\eta, \delta, j}$ converges to the same element $\opPP_{M_\opA(g)} f_0$.
\begin{itemize}
\item[\textbf{(i)}]
Assume $n_*(l)$ to be bounded and set
\begin{equation*}
    N := \max_l n_*(l) \in \mathbb{N}.
\end{equation*}
We observe by the definition of the finite stopping index that
\begin{equation*}
    f_n^{(\eta, \delta, j)_l} = f_N^{(\eta, \delta, j)_l}, \textup{ for all } n\geq N.
\end{equation*}
As a consequence, for those $n\geq N$ the discrepancy principle is satisfied which yields
\begin{equation*}
    \Vert \opA^{\eta_l, j_l} f_n^{(\eta, \delta, j)_l} - g^{\delta_l} \Vert = \Vert \opA^{\eta_l} f_n^{(\eta, \delta, j)_l} - g^{\delta_l} \Vert \leq \tau (\rho' \eta_l + \delta_l) \overset{l\to \infty}{\longrightarrow} 0.
\end{equation*}
On the other hand, we have
\begin{equation*}
    \Vert \opA^{\eta_l,j_l} f_n^{(\eta, \delta, j)_l} - g^{\delta_l} \Vert = \Vert w_n^{(\eta, \delta, j)_l} \Vert  \overset{l\to \infty}{\longrightarrow} \Vert w_n \Vert,
\end{equation*}
according to Corollary \ref{coro:fully_semi_iterates_converge} and Lemma   \ref{lem:induction_helper}. This implies $w_n=0$, so $\opA f_n = g$, which means that $f_n$ is the output of the SESOP algorithm. By Theorem \ref{thm:sesop_convergence} we conclude
\begin{equation*}
    f_n = \opPP_{M_\opA(g)} f_0 \textup{ for all } n\geq N, 
\end{equation*}
which together with Corollary \ref{coro:fully_semi_iterates_converge} leads to 
\begin{equation*}
    f_n^{(\eta, \delta, j)_l} = f_N^{(\eta, \delta, j)_l} \overset{l\to \infty}{\longrightarrow} f_N = \opPP_{M_\opA(g)} f_0.
\end{equation*}

\item[\textbf{(ii)}]
Now assume that $n_*(l)$ is monotonically increasing and unbounded.

On the one hand, the SESOP iterates $(f_n)_n$ converge to $\opPP_{M_\opA(g)}f_0$. On the other hand, by Corollary \ref{coro:not_so_bad_condition} and Corollary \ref{coro:proj_sol_space_conv_closure} we have that
\begin{equation*}
    \opPP_{M_{\opA^j}(g)}f^j_0 \overset{j\to\infty}{\longrightarrow} \opPP_{M_\opA(g)}f_0.
\end{equation*}
Therefore, let choose an $N\in \mathbb{N}$ such that
\begin{eqnarray}
    \Vert f_n - \opPP_{M_\opA(g)}f_0 \Vert < \frac{\varepsilon}{4},\label{eq:thm_regu_1}\\
    \Vert \opPP_{M_{\opA^j}(g)}f^j_0 - \opPP_{M_\opA(g)}f_0 \Vert < \frac{\varepsilon}{4},\label{eq:thm_regu_2}
\end{eqnarray}
for all $n,j \geq N.$ 
According to Corollary \ref{coro:fully_semi_iterates_converge} there exists some $l'\in \mathbb{N}$ with $l'\geq N$ such that 
\begin{equation}
    \Vert f_{N}^{(\eta, \delta, j)_l} - f_{N} \Vert < \frac{\varepsilon}{4},\label{eq:thm_regu_3}
\end{equation}
for all $l\geq l'$. As $n_*(l)$ is unbounded and monotonically increasing we can choose $l''\geq l'$ such that $n_*(l) \geq N$ for all $l\geq l''$. 
For those $l$ we derive
\begin{eqnarray*}\fl
    \Vert f_{n_*(l)}^{(\eta, \delta, j)_l} - \opPP_{M_\opA(g)}f_0\Vert &\leq \Vert f_{n_*(l)}^{(\eta, \delta, j)_l} - \opPP_{M_{\opA^N}(g)}f^N_0  \Vert
    + \Vert \opPP_{M_{\opA^N}(g)}f^N_0 - \opPP_{M_\opA(g)}f_0\Vert\\
    &\overset{(\ref{eq:thm_regu_2})}{<} \frac{\varepsilon}{4} + \Vert f_{n_*(l)}^{(\eta, \delta, j)_l} - \opPP_{M_{\opA^N}(g)}f^N_0 \Vert\\
    &\leq \frac{\varepsilon}{4} + \Vert f_N^{(\eta, \delta, j)_l} - \opPP_{M_{\opA^N}(g)}f^N_0 \Vert,
\end{eqnarray*}
where the last step follows from Proposition \ref{prop:descent_property} b), which is applicable as $\opPP_{M_{\opA^N}(g)}f^N_0$ belongs to the restricted solution set $M^{\rho'}_{\opA^N}(g)$, which is a subset of $M^{\rho'}_{\opA^{n_*(l)}}(g)$.
We further estimate
\begin{eqnarray*}\fl
    \Vert f_N^{(\eta, \delta, j)_l} - \opPP_{M_{\opA^N}(g)}f^N_0 \Vert 
    &\leq \Vert f_N^{(\eta, \delta, j)_l} - f_N \Vert 
    + \Vert f_N - \opPP_{M_{\opA^N}(g)}f^N_0 \Vert \\
    &\overset{(\ref{eq:thm_regu_3})}{<} \frac{\varepsilon}{4} + \Vert f_N - \opPP_{M_{\opA^N}(g)}f^N_0 \Vert\\
    &<  \frac{\varepsilon}{4} + \Vert f_N - \opPP_{M_\opA(g)}f_0 \Vert + \Vert \opPP_{M_\opA(g)}f_0 - \opPP_{M_{\opA^N}(g)}f^N_0 \Vert \\
    &< \frac{3}{4}\varepsilon,
\end{eqnarray*}
due to (\ref{eq:thm_regu_1}) and (\ref{eq:thm_regu_2}).
Altogether, we have shown that for all $l\geq l''$
\begin{equation*}
    \Vert f_{n_*(l)}^{(\eta, \delta, j)_l} - \opPP_{M_\opA(g)}f_0\Vert < \varepsilon.
\end{equation*}

Finally, we have shown convergence of $f_{n_*(l)}^{(\eta, \delta, j)_l}$ to $\opPP_{M_\opA(g)}f_0$ in both cases and hence, by our introductory discussion, in all cases.
\end{itemize}
\end{proof}

We have shown in the last theorem that the RESESOP iterates for $\opA^{\eta,j}$ converge to the SESOP outcome for $\opA$, namely $\opPP_{M_\opA(g)}f_0$. By the choice of $f_0\in N(\opA)^\perp \cap B_{\rho'}(0)$ we have seen in Lemma \ref{lem:generalized_inverse} that $\opPP_{M_\opA(g)}f_0$ must be the minimum-norm solution of $\opA f = g$. As we assumed that $f$ has a solution in $B_\rho(0)$ we conclude that $\opPP_{M_\opA(g)}f_0$ must also be in $B_\rho(0)$ and not only in $B_{\rho'}(0)$ for $\rho'= \max \{ \rho, \rho_j\}$.

\subsection{Application to CST}\label{sec:RESESOP_application}
In this section we specify how the previous framework can be applied to the semi-discrete and fully discrete operators of CST.
We assume that the chosen subspace $\X_j$ of $H^\alpha_0(\Omega)$ ensures surjectivity of the fully discrete operator $(\opL_1^\mu)_j: \X_j\to \R^{P\times K}$ from Section 2.3. We then set 
\begin{eqnarray*}
    \opA^j := (\opL_1^\mu)_j \textup{ \, and \, } \opA^{\eta, j} :=(\opL_1^{\mu^*})_j,
\end{eqnarray*}
as well as
\begin{eqnarray*}
    \opA := \opL_1^\mu \textup{ \, and \, } \opA^\eta :=\opL_1^{\mu^*},
\end{eqnarray*}
for a chosen prior $\mu^*$ to the groundtruth $\mu$. 
If we equip the domain spaces with the Sobolev norm, all operators above are continuous so that the theory in both Sections \ref{sec:RESESOP_restricted} and \ref{sec:RESESOP_restricted_stability} is applicable. That means that applying RESESOP to $(\opL_1^{\mu^*})_j$ does not only regularize $\opL_1^\mu f = g$, see Corollary \ref{coro:regularization_fixed_j}, but is also stable with respect to the chosen subspace in the sense of Theorem \ref{thm:full_semi_regularization}.
However, as mentioned before, we do not have access to the adjoint operators regarding the Sobolev norm, so that for our numerical experiments we need to equip $\X_j$ with the $L_2$-norm. In this case, it still holds that RESESOP applied to $(\opL_1^{\mu^*})_j$ converges to a solution of $\opL_1^\mu f = g$ as soon as the model uncertainty $\eta$ and the noise-level $\delta$ go to zero.
However, this reconstruction may not be stable with respect to the chosen $\X_j$, as the semi-discrete forward operator $\opL_1^\mu$ is no longer continuous.\\

\par Since the model uncertainty highly depends on the respective source, detector positions and energies of incoming photons, we split the operators up by
\begin{equation*}
    \opA^{\eta, j}_{p,k} f := \big((\opL_1^{\mu^*})_j f\big)_{p,k} \in \R, \textup{ for } p\in \{1,...,P\} \textup{ and } k\in \{1,...,K\}
\end{equation*}
and analogously for the other operators. Therefore, in our simulations in Section \ref{sec:simu}, we apply the RESESOP-Kaczmarz Algorithm \ref{algo:resesop} to $\opA^{\eta, j}_{p,k}$.

\section{A Deep Image Prior approach for CST}\label{sec:DIP}

Solving the inverse problem in Compton scattering tomography using standard learning techniques would require large databases obtained from energy-resolved detectors with sufficient energy resolution and $\gamma$-ray sources. Unfortunately, such datasets do not exist preventing the training of a neural network for the CST problem. Alternatively, it is possible to use unsupervised techniques such as Deep Image Prior (DIP), see \cite{Lempitsky2018}.
Therefore, in this section we inspect how the DIP approach can be applied and adapted to the model inexactness in CST.

In this section, we use the same notation as in section \ref{sec:RESESOP_application}. For the sake of readibility, we denote the fully discrete inexact forward operator $(\opL_1^{\mu^*})_j(\cdot)_{p,k}$ by $\opA_{p,k}^{\eta, j}$ and the exact operator $(\opL_1^{\mu})_j(\cdot)_{p,k}$ by $\opA_{p,k}^{j}$.

In order to apply a DIP approach to these operators, we consider a suitable (neural network) mapping $\varphi_\theta: \Z \to \X_j$, where $\theta$ belongs to some parameter space $\Theta$.
Given a single data point $\mathbf{g}^\delta\in \R^{P\times K}$ and some random input $z\in \Z$, the DIP approach seeks for finding parameters $\theta_{opt}\in \Theta$ that minimizes some loss function $\ell$, i.e.
\begin{equation*}
    \theta_{opt} \in \underset{\theta\in \Theta}{\textup{argmin }} \ell\Big(\big(\opA^{\eta,j}_{p,k} \varphi_\theta(z), \mathbf{g}^\delta_{p,k}\big)_{p,k}\Big).
\end{equation*}
The DIP reconstruction is then obtained by evaluating $\varphi_{\theta_{opt}}(z)$.\\

\par Usually, the construction and efficiency of such an approach requires:
\begin{enumerate}
    \item a suitable network architecture must be chosen, which should capture information about the "nature" of images we are looking for, see also \cite{Dittmer_DIP2019};
    \item a stopping criterion in order to avoid noise over-fitting, and
    \item a proper loss function $\ell$, which, as we will see, should also contain information about the model uncertainty between $\opA^{\eta, j}_{p,k}$ and $\opA^j_{p,k}$. In our experiments we focus on inspecting the effect of including model uncertainty estimates to the loss function.
\end{enumerate}

\subsection{Network architecture}
Motivated by the similarities between the model of the first-order scattered and the standard Radon transform, we consider the U-Net provided by J. Leuschner on GitHub\footnote{https://github.com/jleuschn/dival/tree/master/dival/reconstructors/networks}, which was also successfully used in \cite{Baguer_DIP_CT} for CT reconstructions. In our simulation settings, see section \ref{sec:simu}, we obtained the best results for seven layers with $(32,32,64,64,128,128,256)$ channels. For this, we needed to slightly adapt the code, as it was designed for at most six layers. 
Moreover, we used six skip connections and a sigmoid function as a final step of the neural network. 
For the details of the down and up sampling parts of the U-Net we refer directly to the code mentioned above. However, it is reasonable to think that a more optimal network architecture for CST could be constructed in the future, in particular to address the complexity of the model and the multiple-order scattering.

\subsection{Loss functions}
As a first and standard approach - in case of an exact forward operator, see also \cite{Dittmer_DIP2019} - we consider the mean squared error loss function
\begin{equation*}
    \ell_1(\theta) := \frac{1}{PK}\sum_{p,k} \Vert \opA^{\eta,j}_{p,k} \varphi_\theta(z) - \mathbf{g}^\delta_{p,k}\Vert^2,
\end{equation*}
for $\theta\in \Theta$. 
This means that no information on the model uncertainty is explicitly included. By this, we want to inspect, whether the network itself is capable of reducing artifacts caused by considering the inexact forward operators $\opA^{\eta,j}_{p,k}$. As we will see in the next section, this is not the case.

Hence, we want to include information on the model uncertainty to the loss function.
Motivated by the approach in \cite{Blanke_20} to include the model uncertainty to the RESESOP-Kaczmarz procedure, we propose to include the model uncertainty to a loss function via
\begin{equation}\label{def:DIP_loss_eta}
        \ell_2(\theta) := \frac{1}{PK} \sum_{p,k} \Big \vert \Vert \opA^{\eta,j}_{p,k} \varphi_\theta(z) - g^\delta_{p,k} \Vert^2 - c_{p,k}^2 \Big \vert^2,
\end{equation}
where $c_{p,k} := \tau(\rho'\eta_{p,k}+\delta_{p,k})$ is a model correction term inspired from the RESESOP method studied in the previous section. The connection to RESESOP-Kaczmarz is revealed by the following observation: If we assume that one of the summands in (\ref{def:DIP_loss_eta}) is zero, then it is not difficult to see that $\varphi_\theta(z)$ belongs to the boundary of the restricted stripe $B_{\rho'}(0) \cap H(u_{p,k}^{\eta, \delta, j}, \alpha_{p,k}^{\eta, \delta, j}, \xi_{p,k}^{\eta, \delta, j})$, where
\begin{eqnarray*}
u_{p,k}^{\eta, \delta, j} := (\opA^{\eta,j}_{p,k})^*(\opA^{\eta,j}_{p,k}\varphi_\theta(z) - \mathbf{g}^\delta_{p,k}),\\
\alpha_{p,k} := \langle \opA^{\eta,j}_{p,k} \varphi_\theta(z) - \mathbf{g}^\delta_{p,k}, \mathbf{g}^\delta_{p,k} \rangle,\\
\xi^{\eta, \delta, j}_{p,k} := c_{p,k}\Vert \opA^{\eta,j}_{p,k} \varphi_\theta(z) - \mathbf{g}^\delta_{p,k}\Vert 
\end{eqnarray*}
are analogously defined as in Algorithm \ref{algo:resesop}. Hence, if $\theta_{opt}$ is a minimizer of $\ell_2$, then $\varphi_{\theta_{opt}}(z)$ is expected to be close to the boundary of all those stripes. As illustrated in Figure \ref{fig:hyperplane_to_stripe}, the solution of $\opA^j f = \mathbf{g}$ is expected to be close to the stripe boundaries. Further note that $\ell_2$ is also differentiable with respect to $\theta$, given that $\varphi_\theta$ is differentiable, which enables backpropagation.

\begin{remark}
It would also be an option to include $c_{p,k}$ from (\ref{def:DIP_loss_eta}) to the parameter space $\Theta$. However, it is then important to restrict the $c_{p,k}$ to an interval determined by an estimation of the model uncertainty.
This could be achieved in the following way:
Include $\theta^{(c)} \in \mathbb{R}^{P\times K}$ to the parameter space $\Theta$ and choose $c_{p,k}(\theta^{(c)}_{p,k})$ as a differentiable function of $\theta^{(c)}_{p,k}$, whose range is contained in the desired interval. This way, the network could learn a better estimation of the model uncertainty and be less vulnerable to bad model uncertainty estimates. This idea might be valuable for further research.
Moreover, the advantage of considering loss functions like $\ell_2$ is that they probably do not require a stopping criterion.
\end{remark}

We end this section by describing the general training process. For minimizing the loss functions we used the stochastic optimizer \emph{ADAM} from pytorch\footnote[5]{\url{https://pytorch.org/docs/stable/generated/torch.optim.Adam.html}}.
We observed that in the beginning of the training process $\ell_2$ seems to be more sensitive than $\ell_1$ to the choice of the learning rate in the \emph{ADAM} optimizer. More precisely, if the learning rate does not decrease quickly enough during the training, it sometimes happened that the current reconstruction completely changes from one training step to another. 
This might be explained by inspecting the gradients of the loss functions with respect to $\theta$: For simplicity, we consider instead the following functions and their gradients
\begin{eqnarray*}
    f_1(x) := \sum_{k=1}^N \vert x_k\vert^2 \quad &\text{with} \quad  (\nabla f_1(x))_k = 2x_k,\\
    f_2(x) := \sum_{k=1}^N \vert x_k^2 -c_k^2\vert^2 \quad &\text{with} \quad (\nabla f_2(x))_k = 2x_k \cdot 2 ( x_k^2- c_k^2).
\end{eqnarray*}
Thus, if $x$ is not close to $c$ - which is the case in the beginning of the optimization - the gradients of $f_2$ have a larger dynamic than those of $f_1$. So, if the learning rate is not small enough, the gradient descent step for $\ell_2$ is more likely to be too large. 
In order to stabilize the minimization of $\ell_2$ the following strategies turned out to be efficient:
First, starting with $\ell_1$ and later changing the loss function to $\ell_2$ is more robust to the choice of the learning rate.
Second, clipping the gradients during the backpropagation turned out to be another good option to stabilize the loss function $\ell_2$, i.e. rescaling the current gradient as soon as its norm is larger than a chosen threshold. This threshold can iteratively be reduced during the training process. 
In our simulations we combined both approaches.

\section{Simulation results}\label{sec:simu}

In this section, we consider only the two-dimensional case for CST for convenience as the three-dimensional case is significantly more expensive in terms of computations, as mentioned in Section \ref{sec:forward}. However, there is no obstacle in the analysis of both forward models and reconstruction techniques for a direct application to 3D.\\ 

\par We start by presenting the setting of our numerical experiments and exhibit then how the first- and second-order scattering data $\mathbf{g}_1, \mathbf{g}_2 \in \R^{P\times K}$ are simulated. Afterwards, we present the reconstruction results for the RESESOP and the DIP approach.

\begin{description}
    \item \textit{Image domain.} During our experiments the region $\Omega\subset \R^2$ to be scanned is a square of $30$ cm side-length and center at zero. \\
    
    \item \textit{Architecture of the CST scanner.} The detector space $\D$ is a sphere with radius $30$ cm and center at zero. At one half of this sphere $n_s=10$ sources are evenly positioned. For each source we evenly sample $80\%$ of the detector space at $n_d$ locations for the detectors, which is illustrated in Figure \ref{fig:source_detector_setup}.
$20\%$ are omitted because detectors close to the source will only receive a low signal. In total, we have $K=n_s\cdot n_d = 200$ source-detector tuples. \\

\begin{figure}[!htb]
    \centering
    \includegraphics[width=0.45\linewidth]{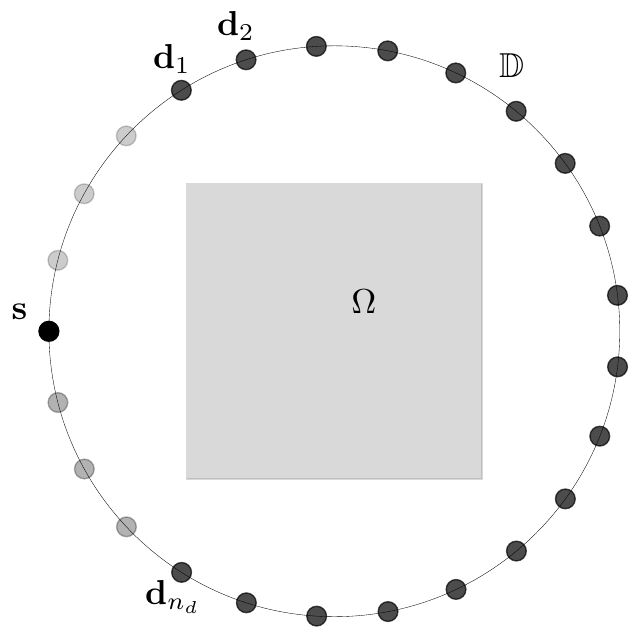}
    \caption{Set-up for first source and corresponding used detectors}
    \label{fig:source_detector_setup}
\end{figure}

    \item \textit{Sources and energy.} The monochromatic sources are assumed to emit $\gamma$-rays at energy $E_0=1173$ keV, which corresponds to the maximal peak of Cobalt-60. Moreover, the total number of emitted photons per source is set to $I_0= 8\cdot 10^8$. Further research and simulations shall take into account the polychromacy of the source as in \cite{KugerRigaud21} but our proposed method can be adapted to this physical aspect.
We also discard the backscattering, \textit{i.e.} scattering angles $\omega \in (\pi/2, \pi)$, as the flux of this part of the spectrum is rather low, thus heavier affected by noise, and further delivers a poorer information for fixed energy resolution, see \cite{Rigaud21}. Therefore,  accordingly with to the Compton formula (\ref{eq:Compton_formula}), we equally sample the energy space $\E$ at $P=80$ energies in the interval $(359.6, 1161.5)$ keV, so for scattering angles $\omega\in (0, \pi/2)$.\\

    \item \textit{A modified Shepp-Logan phantom.} For the groundtruth $\mu$ we consider a bilinear interpolator of a modified Shepp-Logan Phantom defined on a grid twice as fine as $\Omega_h$, so $\mu\notin \X_j$.
    The original Shepp-Logan phantom has a very low contrast in the inner part which is not suited for the level of model inexactness we consider. In order to still provide a challenge for the algorithms, we increased the contrast but not as much as for the "Modified Shepp-Logan" defined in MATLAB.
    The electron densities relative to water of $\mu$ are in the interval $[1.36, 5.66]$, see figure \ref{fig:groundtruth_prior} a). This means its maximal electron density corresponds to bone. Note that the electron density of water is $3.23\cdot 10^{23}$ electrons per cm$^{-3}$.
The horizontal and vertical diameters of the phantom are $19.5$ cm and $26$ cm, respectively. Regarding the prior $\mu^*$, see Figure \ref{fig:groundtruth_prior} b), we choose the same shape as for $\mu$, but set the relative electron density of the interior constantly to $0.67$.
Both $\mu$ and $\mu^*$ are positioned in $\Omega$ such that their center are at zero.\\

    \item \textit{Restricting the domain space.} The finite dimensional subspace $\X_j$ of $H_0^\alpha(\Omega)$ is constructed in the following way: On $\Omega=(-15,15)^2$, we consider a regular $100\times100$ grid $$\Omega_h:=\{ (x_n,y_m)=(-15 + nh, 15+mh)^T: n,m =0,...,99\}$$ for $h=0.3.$ For each grid point, $(x_n,y_m)\in \Omega_h$ we define a gaussian function via 
\begin{equation*}
    e_{nm}(x,y) = c_{nm}\cdot \textup{exp}\left(\frac{1}{2}\left(\frac{x-x_n}{0.5h}\right)^2 +\frac{1}{2}\left( \frac{y-y_m}{0.5h}\right)^2\right),
\end{equation*}
where $c_{nm}$ is chosen such that $\Vert e_{nm}\Vert_{L^2(\Omega)}=1.$ Each gaussian is truncated to the set $\Omega\cap B_r((x_n,y_m)),$ for $r:= 1.5h$.
By setting $\X_j$ as the linear span of the $e_{nm}$ we obtain a $10000$ dimensional space.\\

    \item \textit{Forward models.} For the implementation of the first-order CST operator $\opL_1$, we used the trapezoidal rule for computing the involved integrals. 
The first-order scattering data $\mathbf{g}_1\in \R^{P\times K}$ is then computed by evaluating the semi-discrete $\opL_1^\mu$ at $\mu$ for the respective detector-source-energy triples described above.
The second-order scattering data $\mathbf{g}_2\in \R^{P\times K}$ was generated using Monte-Carlo simulations \cite{Rigaud21}. Since "only" $I_0 = 8\cdot 10^8$ photons were sent by the source, the second-order scattered radiation is subject to noise due to the stochastic nature of the emission of photons.  
The data $\mathbf{g}_1,$ $\mathbf{g}_2$ and the sum $\mathbf{g}_1+\mathbf{g}_2$ for one source position are depicted in Figure \ref{fig:shepp_data}.\\

\begin{figure}[t]
     \centering
     \begin{subfigure}[]
         \centering
         \includegraphics[width=0.23\linewidth]{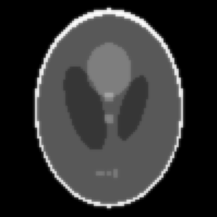}
     \end{subfigure}
     \begin{subfigure}[]
         \centering
         \includegraphics[width=0.23\linewidth]{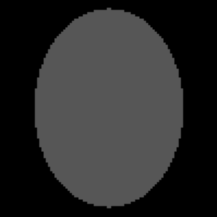}
     \end{subfigure}
     
     \begin{subfigure}[\label{fig:shepp_data}]
         \centering 
          \includegraphics[width=\linewidth]{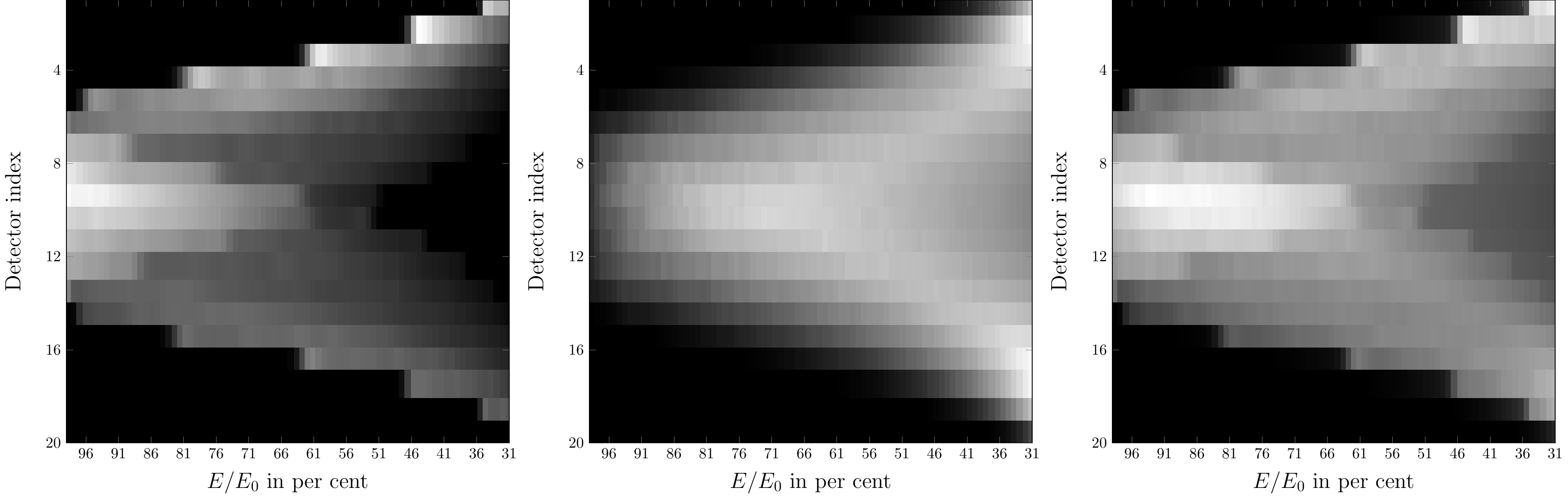}
    \end{subfigure}
     \caption{(a) Shepp-Logan phantom (groundtruth), (b) Prior map used for $\mu^*$, (c) CST data for one source position -- from left to right: $\mathbf{g}_1$, $\mathbf{g}_2$ and $\mathbf{g}_1+\mathbf{g}_2$.}
    \label{fig:groundtruth_prior}
\end{figure}

    \item \textit{Inexact model.} Regarding the inexact fully discrete operator $(\opL_1^{\mu^*})_j:\X_j\to \R^{P\times K}$, we compute its matrix representation $(P^{\mu^*}_1)_j$ - an $16000\times 10000$ matrix - with respect to the basis $(e_{nm})_{n,m}$ and the standard basis of $\R^{P\times K}$, that is, its columns are flattened versions of the vectors $\opL_1^{\mu^*}(e_{nm})\in \R^{P\times K}$. This allows a fast evaluation of the fully-discrete operator.
For computing the matrix entries, we pre-computed the weight $\opW_1(\mu^*)$ on a grid twice as fine as $\Omega_h$ and used linear interpolation. Thereby, the computation time was reduced.\\
    
    \item \textit{Different scenarios.} In our numerical experiments with RESESOP and DIP, we consider the following  reconstruction problems and use the notation introduced in the previous sections.
\begin{itemize}
    \item[\textbf{(i)}] $\opA_{p,k}^{\eta, j}f = (\mathbf{g}_1)_{p,k}$ for $p=1,\ldots P,$ $k=1,\ldots K$;
    \item[\textbf{(ii)}] $\opA_{p,k}^{\eta, j}f = (\mathbf{g}_1^\delta)_{p,k}$, where $\mathbf{g}_1^\delta$ is $\mathbf{g}_1$ disturbed by $2.4\%$ Poisson noise;
    \item[\textbf{(iii)}] $\opA_{p,k}^{\eta, j}f = (\mathbf{g}_1+\mathbf{g}_2)_{p,k}$ where $\mathbf{g}_2$ is corrupted by Poisson noise due to the Monte-Carlo process;
    \item[\textbf{(iv)}] $(\mathfrak{P} \opA)_{p,k}^{\eta^\mathfrak{P}, j}f = (\mathfrak{P}(\mathbf{g}_1+\mathbf{g}_2))_{p,k}$, where $\mathfrak{P}$ computes finite differences of the input vector.
\end{itemize}
\end{description}

\par 
For our reconstructions we need accurate estimations of the model uncertainty for every subproblem. We computed them numerically by inspecting the discrepancy between data generated by the exact and inexact forward operators.
Further, we use four different similarity measurements in Table \ref{tab:comparison-results_justg1},\ref{tab:comparison-results_justg1_noisy} and \ref{tab:comparison-results_g2} for comparing the different reconstructions, namely:
Signal-to-Noise ratio (SNR) \footnote[1]{SNR computed as defined in \url{https://github.com/scipy/scipy/blob/v0.16.0/scipy/stats/stats.py#L1963}}, peak signal-to-noise ratio (PSNR), structural self-similarity (SSIM) and normalized mean square error (NMSE).\\

\par 
In scenario \textbf{(i)}, that is dealing only with exact first-order scattering data $\mathbf{g}_1$, we present the outcome of six reconstruction methods. Three of them, Landweber with early stopping, Total Variation (TV) - see Remark \ref{rem:TV} - and the DIP approach with loss function $\ell_1$, do not take the model uncertainty into account and are depicted in Figure \ref{fig:g1_recos} (a), (c) and Figure \ref{fig:DIP_recos} (a), respectively. We observe that the overall contours of different tissues are well recognizable, which is expectable according to the results in \cite{Rigaud21}. However, the contrast between different tissues is badly reconstructed. 
In comparison to that, inspecting the RESESOP-Kaczmarz reconstruction in Figure \ref{fig:g1_recos} (e), the contrast between different tissues is much better retained. But a certain noisy pattern is noticeable in this reconstruction, which might be caused by considering multiple inverse problems instead of one. To deal with this problem one could e.g. do some post-processing with a suitable denoiser. Empirically, a good strategy was to include a few TV denoising steps after every $100$ RESESOP-Kaczmarz sweeps. By that we obtained satisfactory results, see Figure \ref{fig:DIP_recos} (g). Nevertheless, it is to be noted here, that it is probably not easy to prove convergence for a combination of RESESOP and a non-projective method like TV-denoising.
Finally, the DIP reconstruction with loss function $\ell_2$ is depicted in Figure \ref{fig:DIP_recos} (b) and looks pretty similar to the RESESOP reconstruction, which is not very surprising, as the discrepancy term $c_{p,k}$ in $\ell_2$ was chosen like in the discrepancy principle (\ref{eq:discrepancy_principle}) in the RESESOP Algorithm \ref{algo:resesop}.
The reconstruction errors are listed in Table  \ref{tab:comparison-results_justg1} and the best results were achieved by the the combination of RESESOP and TV-denoising, called RESESOP+TV in the following.

\begin{remark}\label{rem:TV}
Introduced in \cite{ROF}, the Total-Variation has become a standard for regularizing inverse problems in imaging and image processing and is solved by the primal-dual algorithm. Here, we considered the regularized TV introduced and analyzed in \cite{acar94}, \textit{i.e.}
for solving an inverse problem $\opB f = g^\delta$ we find a minimizer of 
\begin{equation*}
    \Vert \opB f - g^\delta\Vert_{L^2(\Omega)}^2 + \lambda \int_\Omega \sqrt{\vert \nabla f \vert^2 +\beta}, \qquad \beta>0.
\end{equation*}
The objective function is differentiable, so gradient methods can be used for deriving a solution. For the TV denoising step in RESESOP+TV, we chose $\opB$ to be the identity operator.
\end{remark}

In scenario \textbf{(ii)} for corrupted data $\mathbf{g}_1^\delta$ we observe that RESESOP, RESEOP+TV and DIP with loss function $\ell_2$ are able to handle both the model uncertainty and noise if good approximations of the noise-levels $\delta_{p,k}$ are known, see Figure \ref{fig:g1_recos} (f), (h) and Figure \ref{fig:DIP_recos} (c). However, the corresponding reconstruction errors in Table \ref{tab:comparison-results_justg1_noisy} are a bit worse than in case of noise-free data in Table \ref{tab:comparison-results_justg1}. Except for RESESOP+TV. Here the reconstruction remains good, even a slightly better, which might be because of a better choice by hand of the TV-denoising parameter.
In the Landweber and TV reconstructions in Figure \ref{fig:g1_recos} (b) and (d) some details are gone.\\

\par 
Incorporating the second-order scattering data $\mathbf{g}_2$ to $\mathbf{g}_1$ in scenario \textbf{(iii)} leads to a huge additional uncertainty, as the flux of $\mathbf{g}_2$ is almost as high as the flux of $\mathbf{g}_1$, that is $\Vert \mathbf{g}_2 \Vert_1\approx \Vert \mathbf{g}_1 \Vert_1$, where
\begin{equation*}
    \Vert x \Vert_1 := \sum_{p,k} \vert x_{p,k} \vert, \textup{ for } x\in \R^{P\times K}.
\end{equation*} 
In this case, we see that RESESOP is no longer capable to handle it, see Figure \ref{fig:g2_Pg2_recos} (d). Also TV in Figure \ref{fig:g2_Pg2_recos} (a) is no longer able to reconstruct the inner contours. 

Therefore, we applied a finite difference operator $\mathfrak{P}$ in scenario \textbf{(iv)} to both sides of the problem $\opA^{\eta, j} f = (\mathbf{g}_1+\mathbf{g}_2)$. This reduced the flux of both first- and second-order data, but more importantly, the latter decreased more: $\Vert \mathfrak{P} \mathbf{g}_2 \Vert\leq 0.44 \cdot \Vert \mathfrak{P} \mathbf{g}_1 \Vert$. Indeed, by this, RESESOP and RESESOP+TV lead to good reconstructions, see Figure \ref{fig:g2_Pg2_recos} (e) and (f). However, the contrast between different tissues is not as well reconstructed as in the case for just $\mathbf{g}_1$ data, see also Table \ref{tab:comparison-results_g2} for the reconstruction errors. The DIP approach with loss function $\ell_2$ in Figure \ref{fig:DIP_recos} (d) has some artifacts in form of scratches, but the contrast between tissues is better conserved, which results in comparable reconstructions errors to RESESOP and RESESOP+TV in Table \ref{tab:comparison-results_g2}.

\begin{remark}
As also the DIP reconstructions in Figure \ref{fig:DIP_recos} include some noisy pattern, we tried to add a further denoising penalty to the loss function $\ell_2$ in the DIP approach. Unsurprisingly, thereby the influence of the model uncertainty gets more visible again. Therefore, we propose to gain improvements rather by some post-processing, for example by TV denoising.
\end{remark}

\begin{remark}
The prior $\mu^*$ is very simple, so the model uncertainty is rather large. To decrease the model uncertainty, one could use one of the reconstructions as a new prior $\tilde{\mu}^*$ and consider $\opA^\eta = \opL_1^{\tilde{\mu}^*}$, which probably is a better approximation of $\opL_1^\mu$. Furthermore, a prior $\mathbf{g}_2^*$ could be included in the discrepancy term, both for RESESOP and DIP, in order to reduce the model uncertainty. We did not consider these improvements to stress the algorithms in terms of model uncertainty.
\end{remark}

\begin{table}[h]
	    \centering
	    \setlength{\tabcolsep}{3pt}
	    \renewcommand{\arraystretch}{1.1}
	\begin{tabular}{c|cccccc}
	     & Landweber & TV($\lambda=12$) & DIP $\ell_1$ & RESESOP & RESESOP+TV & DIP $\ell_2$ \\\hline
	SNR & 0.877 & 0.706 & 0.539 & 0.903 & \textbf{0.561} & 0.647 \\
	PSNR (dB) & 22.305 & 24.083 & 23.667 & 27.570 & \textbf{33.541} & 28.990 \\
	SSIM & 0.948 & 0.965 & 0.962 & 0.985 & \textbf{0.996} & 0.989 \\
	NMSE & 0.249 & 0.203 & 0.212 & 0.136 & \textbf{0.068} & 0.115 \\
	\end{tabular}
	\caption{Scenario \textbf{(i)}: Error measures for the different reconstructions and methods to solve $\opA^{\eta, j}f =\mathbf{g}_1$. }
	\label{tab:comparison-results_justg1}
\end{table}

\begin{table}[h]
	    \centering
	    \setlength{\tabcolsep}{3pt}
	    \renewcommand{\arraystretch}{1.1}
	\begin{tabular}{c|ccccc}
	     & Landweber & TV($\lambda=45$) &  RESESOP & RESESOP+TV & DIP $\ell_2$ \\\hline
	SNR & 0.891 & 0.616 & 0.941 & \textbf{0.602} & 0.632 \\
	PSNR (dB) & 21.502 & 22.772 & 26.651 & \textbf{33.978} & 31.804 \\
	SSIM & 0.937 & 0.953 & 0.982 & \textbf{0.997} & 0.995  \\
	NMSE & 0.273 & 0.236 & 0.151 & \textbf{0.065} & 0.083 \\
	\end{tabular}
	\caption{Scenario \textbf{(ii)}: Error measures for the different reconstructions and methods to solve $\opA^{\eta, j}f =\mathbf{g}_1^\delta$, i.e. data with $2.4\%$ Poisson noise. }
	\label{tab:comparison-results_justg1_noisy}
\end{table}

\begin{table}[h]
	    \centering
	    \setlength{\tabcolsep}{3pt}
	    \renewcommand{\arraystretch}{1.1}
	\begin{tabular}{c|ccccc}
	     & TV($\lambda=9$) & TV($\lambda=26$) & RESESOP & RESESOP+TV & DIP $\ell_2$\\\hline
	SNR & 0.967           & 0.904            & 0.815   & 0.659      & \textbf{0.545} \\
	PSNR (dB) & 20.657          & 21.056           & 23.788  & \textbf{25.804}     & 24.785 \\
	SSIM & 0.928           & 0.933            & 0.963   & \textbf{0.977}      & 0.972 \\
	NMSE & 0.300           & 0.287            & 0.210   & \textbf{0.166}      & 0.187
	\end{tabular}
	\caption{Scenario \textbf{(iv)}: Error measures for the different reconstructions and methods to solve $(\opP \opA)^{\eta^\mathfrak{P}, j}f = \opP(\mathbf{g}_1+\mathbf{g}_2)$. }
	\label{tab:comparison-results_g2}
\end{table}

\begin{figure}[!htb]
     \centering
     \begin{subfigure}[]
         \centering
         \includegraphics[width=0.23\linewidth]{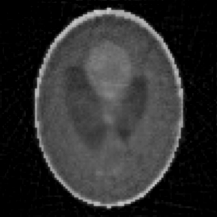}
     \end{subfigure}
     \begin{subfigure}[]
         \centering
         \includegraphics[width=0.23\linewidth]{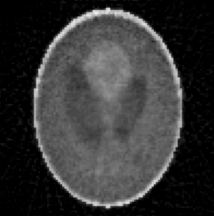}
     \end{subfigure}
     \begin{subfigure}[]
         \centering
         \includegraphics[width=0.23\linewidth]{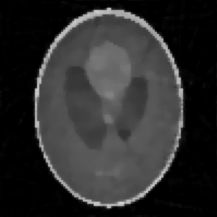}
     \end{subfigure}
     \begin{subfigure}[]
         \centering
         \includegraphics[width=0.23\linewidth]{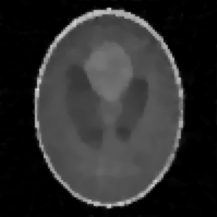}
     \end{subfigure}

     \begin{subfigure}[]
         \centering
         \includegraphics[width=0.23\linewidth]{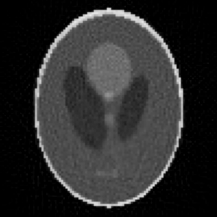}
     \end{subfigure}
     \begin{subfigure}[]
         \centering
         \includegraphics[width=0.23\linewidth]{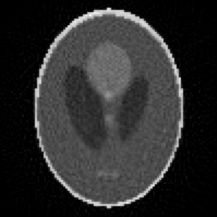}
     \end{subfigure}
     \begin{subfigure}[]
         \centering
         \includegraphics[width=0.23\linewidth]{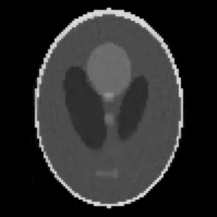}
     \end{subfigure}
     \begin{subfigure}[]
         \centering
         \includegraphics[width=0.23\linewidth]{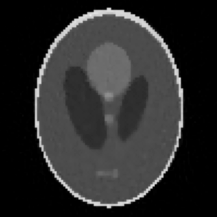}
     \end{subfigure}
     \caption{Reconstructions for  \\Scenario \textbf{(i)}: (a) Landweber, (c) TV, (e) RESESOP, (g) RESESOP+TV\\Scenario \textbf{(ii)}: (b) Landweber, (d) TV, (f) RESESOP, (h) RESESOP+TV\\}
    \label{fig:g1_recos}
\end{figure}

 \begin{figure}
     \centering
     \begin{subfigure}[]
         \centering
         \includegraphics[width=0.23\linewidth]{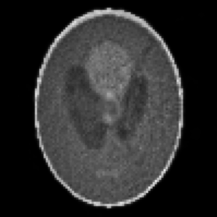}
     \end{subfigure}
     \begin{subfigure}[]
         \centering
         \includegraphics[width=0.23\linewidth]{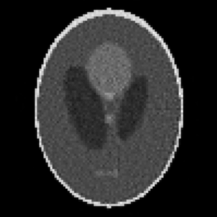}
     \end{subfigure}
     \begin{subfigure}[]
         \centering
         \includegraphics[width=0.23\linewidth]{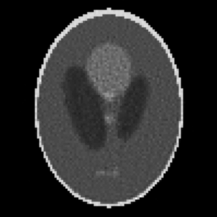}
     \end{subfigure}
          \begin{subfigure}[]
         \centering
         \includegraphics[width=0.23\linewidth]{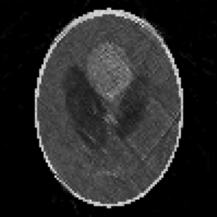}
     \end{subfigure}
     \caption{DIP reconstructions for\\ Scenario \textbf{(i)}: (a) with $\ell_l$, (b) with $\ell_2$\\ Scenario \textbf{(ii)}: (c) with $\ell_2$\\
     Scenario \textbf{(iv)}: (d)  with $\ell_2$\\}
    \label{fig:DIP_recos}
\end{figure}

\begin{figure}[!htb]
     \centering
     \begin{subfigure}[]
         \centering
         \includegraphics[width=0.23\linewidth]{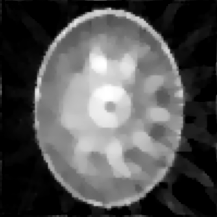}
     \end{subfigure}
     \begin{subfigure}[]
         \centering
         \includegraphics[width=0.23\linewidth]{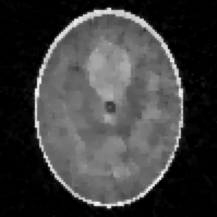}
     \end{subfigure}
     \begin{subfigure}[]
         \centering
         \includegraphics[width=0.23\linewidth]{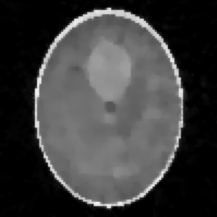}
     \end{subfigure}

     \begin{subfigure}[]
         \centering
         \includegraphics[width=0.23\linewidth]{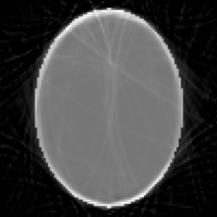}
     \end{subfigure}
     \begin{subfigure}[]
         \centering
         \includegraphics[width=0.23\linewidth]{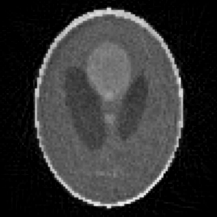}
     \end{subfigure}
     \begin{subfigure}[]
         \centering
         \includegraphics[width=0.23\linewidth]{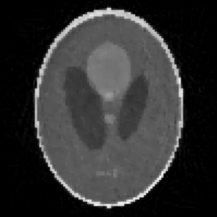}
     \end{subfigure}
     \caption{Reconstructions for \\Scenario \textbf{(iii)}: (a) TV, (d) RESESOP\\  Scenario \textbf{(iv)}: (b,c) TV ($\lambda=9$ and 26),
     (e) RESESOP, (f) RESESOP+TV}
    \label{fig:g2_Pg2_recos}
\end{figure}

\section{Conclusion}

We have proposed two data-driven reconstruction strategies able to handle the model uncertainty occurring in imaging based on Compton scattering. The construction of these algorithms is based on the study of the properties of the forward models: nonlinearity, mapping properties and model uncertainty. The first approach considers the RESESOP method which is studied in terms of convergence and regularization for the fully discrete case in order to fit the restrictions of our spectral inverse problem. The second approach exploits the popular DIP method, suited for the treated problem since unsupervised, it does not require a dataset. We modified the learning loss function using the model uncertainty model used in the first approach. Simulation results on synthetic data for the first-order scattering and on Monte-Carlo data for the second-order scattering attest the efficiency of both approaches. \\

\par The performed simulations assumed an almost perfect estimation of the model uncertainty for every subproblems which is hard to achieve in practice and remains an open issue for the general RESESOP approach or here for our modified DIP method. A first possibility would be to learn the model uncertainty coefficients from a synthetic dataset or from real dataset in the future. Another more general approach would be to relax the uncertainty parameter in the RESESOP method, for instance by incorporating a minimization problem at each iterate to find the best parameter. These questions will be the core of future research.
\section*{Acknowledgements}
This research was supported by the Deutsche Forschungsgemeinschaft (DFG) under the grant RI 2772/2-1 and the Cluster of Excellence EXC 2075 "Data-Integrated Simulation Science" at the University of Stuttgart.

\section*{References}

\bibliographystyle{siamplain}
\bibliography{references}

\end{document}